\documentclass[11pt,reqno,draft]{amsart}

\usepackage[a4paper,margin=30mm]{geometry}

\usepackage{verbatim,upref,amsxtra,amssymb,amscd,graphicx}

\usepackage{mathtools}

\usepackage{color}

\usepackage[mathscr]{eucal}

\usepackage[english]{babel}

\numberwithin{equation}{section}

\newtheorem{theo}{Theorem}[section]
\newtheorem{prop}[theo]{Proposition}
\newtheorem{lemm}[theo]{Lemma}
\newtheorem{coro}[theo]{Corollary}

\theoremstyle{definition}

\newtheorem{defi}[theo]{Definition}
\newtheorem{exam}[theo]{Example}

\theoremstyle{remark}

\newtheorem{rema}[theo]{Remark}

\DeclareMathOperator{\N}{\mathbb{N}}
\DeclareMathOperator{\Z}{\mathbb{Z}}
\DeclareMathOperator{\Q}{\mathbb{Q}}
\DeclareMathOperator{\R}{\mathbb{R}}
\DeclareMathOperator{\C}{\mathbb{C}}
\DeclareMathOperator{\T}{\mathbb{T}}

\DeclareMathOperator{\h}{\mathbb{H}}
\DeclareMathOperator{\s}{\mathbb{S}} 
\def\qbinom #1#2{{\genfrac{[}{]}{0pt}{}{#1}{#2}}_{q}}
\def\qbinomchi #1#2{{\genfrac{[}{]}{0pt}{}{#1}{#2}}_{\chi^{-1}}}

\newcommand{\norm}[1]{\left \| #1 \right \| _1}

\newcommand{\dT}{\ensuremath|\hspace{-2.1pt}|}

\begin{document}
\allowdisplaybreaks\frenchspacing

\baselineskip=1.2\baselineskip

\title[Algebraic actions of the discrete Heisenberg group] {Algebraic actions of the discrete Heisenberg group: Expansiveness and homoclinic points}

\author{Martin G\"oll} \address{Martin G\"oll: Mathematical Institute, University of Leiden, PO Box 9512, 2300 RA Leiden, The Netherlands} \email{gollm@math.leidenuniv.nl}

\author{Klaus Schmidt}
\address{Klaus Schmidt: Mathematics Institute, University of Vienna, Oskar-Morgenstern-Platz 1, A-1090 Vienna, Austria \newline\indent \textup{and} \newline\indent Erwin Schr\"odinger Institute for Mathematical Physics, Boltzmanngasse~9, A-1090 Vienna, Austria} \email{klaus.schmidt@univie.ac.at}

\author{Evgeny Verbitskiy}
\address{Evgeny Verbitskiy: Mathematical Institute, University of Leiden, PO Box 9512, 2300 RA Leiden, The Netherlands \newline\indent \textup{and} \newline\indent Johann Bernoulli Institute for Mathematics and Computer Science, University of Groningen, PO Box 407, 9700 AK, Groningen, The Netherlands} \email{e.a.verbitskiy@rug.nl}

\thanks{The authors would like to thank Hanfeng Li and Doug Lind for helpful discussions and insights.\newline\indent MG gratefully acknowledges support by a
Huygens Fellowship from Leiden University.
\newline\indent MG and EV would like to thank the Erwin Schr\"{o}dinger
Institute, Vienna, and KS the University of Leiden, for hospitality and support while some of this work was done.
}

	\begin{abstract}
We survey some of the known criteria for expansiveness of principal algebraic actions of countably infinite discrete groups. In the special case of the discrete Heisenberg group we propose a new approach to this problem based on Allan's local principle.

Furthermore, we present a first example of an absolutely summable homoclinic point for a nonexpansive action of the discrete Heisenberg group and use it to construct an equal-entropy symbolic cover of the system.
	\end{abstract}
	\maketitle

\section{Introduction}

Let $\Gamma$ be a countably infinite discrete group with integer group ring $\mathbb{Z}[\Gamma ]$. Every $g\in\mathbb{Z}[\Gamma ]$ is written as a formal sum $g =\sum_{\gamma }g_{\gamma }\cdot \gamma $, where $g_{\gamma }\in\mathbb{Z}$ for every $\gamma \in\Gamma $ and $\sum_{\gamma \in\Gamma }|g_{\gamma }|<\infty $. The set $\textup{supp}(g)=\{\gamma \in \Gamma :g_\gamma \ne 0\}$ is called the \textit{support} of $g$. For $g=\sum_{\gamma \in\Gamma }g_\gamma \cdot \gamma \in\mathbb{Z}[\Gamma ]$ we denote by $g^*=\sum_{\gamma \in\Gamma }g_\gamma \cdot \gamma ^{-1}$ the \textit{adjoint} of $g$. The map $g\mapsto g^*$ is an \textit{involution} on $\mathbb{Z}[\Gamma ]$, i.e., $(gh)^*=h^*g^*$ for all $g,h\in\mathbb{Z}[\Gamma ]$, where the product $fg$ of two elements $f=\sum_{\gamma }f_\gamma \cdot \gamma $ and $g=\sum_\gamma g_\gamma \cdot \gamma $ in $\mathbb{Z}[\Gamma ]$ is given by \textit{convolution}: 
$$fg=\sum_{\gamma ,\gamma '\in\Gamma }f_\gamma g_{\gamma '}\cdot \gamma \gamma '=\sum_{\gamma \in \Gamma }\sum_{\delta \in \Gamma }f_\gamma g_{\gamma ^{-1}\delta }\cdot \delta \,.$$
In view of this it will occasionally be convenient to write $f*g$ instead of $fg$ for the product in $\mathbb{Z}[\Gamma ]$.

An \textit{algebraic $\Gamma $-action} is a homomorphism $\alpha \colon \Gamma \longrightarrow \textup{Aut}(X)$ from $\Gamma $ to the group of (continuous) automorphisms of a compact second countable abelian group $X$. If $\alpha $ is an algebraic $\Gamma $-action, then $\alpha ^\gamma \in \textup{Aut}(X)$ denotes the image of $\gamma \in \Gamma $, and $\alpha ^{\gamma \gamma '}=\alpha ^\gamma \circ \alpha ^{\gamma '}$ for every $\gamma ,\gamma '\in \Gamma $. The action $\alpha $ induces an action of $\mathbb{Z}[\Gamma ]$ by group homomorphisms $\alpha ^f\colon X\longrightarrow X$, where $\alpha ^f=\sum_{\gamma \in\Gamma }f_\gamma \alpha ^\gamma $ for every $f=\sum_{\gamma \in\Gamma }f_\gamma \cdot \gamma \in\mathbb{Z}[\Gamma ]$. Clearly, if $f,g\in\mathbb{Z}[\Gamma ]$, then $\alpha ^{fg}=\alpha ^f\circ \alpha ^g$.

Let $\hat{X}$ be the dual group of the compact abelian group $X$. For every $\gamma \in \Gamma $ we denote by $\hat{\alpha }^\gamma $ the automorphism of $\hat{X}$ dual to $\alpha ^\gamma $ and observe that $\hat{\alpha }^{\gamma \gamma '}=\hat{\alpha }^{\gamma '}\circ \hat{\alpha }^\gamma $ for all $\gamma ,\gamma '\in \Gamma $. If $\hat{\alpha }^f\colon \hat{X}\longrightarrow \hat{X}$ is the group homomorphism dual to $\alpha ^f$ we set $f\cdot a=\hat{\alpha }^{f^*}a$ for every $f\in\mathbb{Z}[\Gamma ]$ and $a\in \hat{X}$. The resulting map $(f,a)\mapsto f\cdot a$ from $\mathbb{Z}[\Gamma ]\times \hat{X}$ to $\hat{X}$ satisfies that $(fg)\cdot a=f\cdot (g\cdot a)$ for all $f,g\in\mathbb{Z}[\Gamma ]$ and turns $\hat{X}$ into a module over the group ring $\mathbb{Z}[\Gamma ]$. Conversely, if $M$ is a countable module over $\mathbb{Z}[\Gamma ]$, we set $X=\widehat{M}$ and put $\hat{\alpha }^fa=f^*\cdot a$ for $f\in\mathbb{Z}[\Gamma ]$ and $a\in M$. The maps $\alpha ^f\colon \widehat{M}\longrightarrow \widehat{M}$ dual to $\hat{\alpha }^f,\,f\in\mathbb{Z}[\Gamma ]$, define an action of $\mathbb{Z}[\Gamma ]$ by homomorphisms of $\widehat{M}$, which in turn induces an algebraic action $\alpha $ of $\Gamma $ on $X=\widehat{M}$.

The simplest examples of algebraic $\Gamma $-actions arise from $\mathbb{Z}[\Gamma ]$-modules of the form $M=\mathbb{Z}[\Gamma ]/\mathbb{Z}[\Gamma ]f$ with $f\in\mathbb{Z}[\Gamma ]$. Since these actions are determined by principal left ideals of $\mathbb{Z}[\Gamma ]$ they are called \textit{principal algebraic $\Gamma $-actions}. In order to describe these actions more explicitly we put $\mathbb{T}=\mathbb{R}/\mathbb{Z}$ and define the left and right shift-actions $\lambda $ and $\rho $ of $\Gamma $ on $\mathbb{T}^\Gamma $ by setting
	\begin{equation}
	\label{eq:lambda}
(\lambda ^\gamma x)_{\gamma '}=x_{\gamma ^{-1}\gamma '},\qquad (\rho ^\gamma x)_{\gamma '}=x_{\gamma '\gamma },
	\end{equation}
for every $\gamma \in \Gamma $ and $x=(x_{\gamma '})_{\gamma '\in \Gamma }\in\mathbb{T}^\Gamma $. The $\Gamma $-actions $\lambda $ and $\rho $ extend to actions of $\mathbb{Z}[\Gamma ]$ on $\mathbb{T}^\Gamma $ given by
	\begin{equation}
	\label{eq:Lambda}
\lambda ^f=\textstyle\sum_{\gamma \in\Gamma }f_\gamma \lambda ^\gamma ,\qquad \rho ^f=\textstyle\sum_{\gamma \in\Gamma }f_\gamma \rho ^\gamma
	\end{equation}
for every $f= \sum_{\gamma \in\Gamma }f_\gamma \cdot \gamma \in\mathbb{Z}[\Gamma ]$.

\smallskip The pairing $\langle f,x\rangle =e^{2\pi i\sum_{\gamma \in \Gamma }f_\gamma x_\gamma }$, $f=\sum_{\gamma \in \Gamma }f_\gamma \cdot \gamma \in\mathbb{Z}[\Gamma ]$, $x=(x_\gamma )\in\mathbb{T}^\Gamma $, identifies $\mathbb{Z}[\Gamma ]$ with the dual group $\widehat{\mathbb{T}^\Gamma }$ of $\mathbb{T}^\Gamma $. We claim that, under this identification,
	\begin{equation}
	\label{eq:Xf}
	\begin{aligned}
X_f&\coloneqq\ker \rho ^f = \bigl\{x\in \mathbb{T}^\Gamma :\rho ^fx=\textstyle\sum_{\gamma \in \Gamma }f_\gamma \rho ^\gamma x=0\bigr\}
	\\
&\,=(\mathbb{Z}[\Gamma ]f)^\perp = \widehat{\mathbb{Z}[\Gamma ]/\mathbb{Z}[\Gamma ]f}\subset \widehat{\mathbb{Z}[\Gamma ]}=\mathbb{T}^\Gamma .
	\end{aligned}
	\end{equation}
Indeed,
	\begin{align*}
\langle h,\rho ^fx\rangle &=\smash[t]{\Bigl\langle h,\sum\nolimits_{\gamma '\in \Gamma }f_{\gamma '}\rho ^{\gamma '}x\Bigr\rangle =\sum\nolimits_{\gamma \in \Gamma }h_\gamma \sum\nolimits_{\gamma '\in \Gamma }f_{\gamma '}x_{\gamma \gamma '}}
	\\
&=\sum\nolimits_{\gamma \in \Gamma }\sum\nolimits_{\gamma '\in \Gamma }h_{\gamma \gamma '^{-1}}f_{\gamma '}x_\gamma =\sum\nolimits_{\gamma \in \Gamma }(hf)_\gamma x_\gamma =\langle hf,x\rangle
	\end{align*}
for every $h\in\mathbb{Z}[\Gamma ]$ and $x\in\mathbb{T}^\Gamma $, so that $x\in\ker \rho ^f$ if and only if $x\in (\mathbb{Z}[\Gamma ]f)^\perp$.

Since the $\Gamma $-actions $\lambda $ and $\rho $ on $\mathbb{T}^\Gamma $ commute, the group $X_f=\ker\rho ^f\subset \mathbb{T}^\Gamma $ is invariant under $\lambda $, and we denote by $\alpha _f$ the restriction of $\lambda $ to $X_f$. For convenience of terminology we introduce the following definition.

	\begin{defi}
	\label{d:principal}
$(X_f,\alpha _f)$ is the \textit{principal algebraic $\Gamma $-action} defined by $f\in\mathbb{Z}[\Gamma ]$.
	\end{defi}

An algebraic action $\alpha $ of a countable group $\Gamma $ on a compact abelian group $X$ with identity element $0=0_X$ and an invariant metric $d$ is called \textit{expansive} if there exists an $\varepsilon >0$ such that
	\begin{equation}
	\label{eq:expansive}
\sup_{\gamma \in \Gamma} d (\alpha^\gamma x, \alpha^\gamma y) = \sup_{\gamma \in \Gamma} d \bigl(\alpha^\gamma (x-y), 0_X\bigr)\geq \varepsilon
	\end{equation}
for all distinct $x,y \in X$. According to \cite[Theorem 3.2]{DS}, expansiveness of a principal action $\alpha _f$, $f\in\mathbb{Z}[\Gamma ]$, is equivalent to the invertibility of $f \in \Z [ \Gamma ]$ in the larger group algebra $\ell^1(\Gamma,\mathbb{R})$. Unfortunately this characterisation of expansiveness does not --- in general --- lend itself to an effective `algorithmic' test for expansiveness. Only in special cases, like for $\Gamma =\Z^d$ (\cite[Lemma 6.8 and Theorem 6.5 (4)]{Schmidt} or, more generally, for every countable abelian group (\cite[Theorem 3.3]{Miles} in combination with Wiener's Lemma), one obtains an easily verifiable criterion to decide if an algebraic dynamical system is expansive or not.

For nonabelian countably infinite groups, checking invertibility of an element $f\in \mathbb{Z}[\Gamma ]$ in the corresponding group algebra $\ell^1(\Gamma,\mathbb{R})$ is far from simple. In this paper we shall concentrate on the discrete Heisenberg group $\Gamma =\mathbb{H}\subset \textup{SL}(3,\mathbb{Z})$, defined by
	\begin{equation}
	\label{eq:Heisenberg}
\mathbb{H} = \left\{\left(
	\begin{smallmatrix}
1 & a & b
	\\
0 & 1 & c
	\\
0 & 0 & 1
	\end{smallmatrix}
\right)\colon\enspace a,b,c \in \Z \right\}.
	\end{equation}
The discrete Heisenberg group $\mathbb{H}$ is not of type I, since it does not possess an abelian normal subgroup of finite index (cf. \cite{Thoma}). Therefore, its spectrum, i.e.\ the space of unitary equivalence classes of irreducible unitary representations of $\mathbb{H}$, has no nice parametrisation (it is not a standard Borel space). This and several other pathologies which result from the fact that $\h$ is not of type I are discussed in \cite[Chapter 7]{Folland}.

In this paper we discuss several methods to decide whether an element $f \in \Z[\h]$ is invertible in $\ell^1(\h,\C)$ or not. In Section 3 we introduce Allan's local principle, which can be summarised as follows: in order to study invertibility of an element $f$ in $\ell^1(\h,\mathbb{R})$ we view $f$ as an element of the complexification $\ell^1(\h,\mathbb{C})$ of $\ell ^1(\mathbb{H},\mathbb{R})$ and project it to its equivalence classes $ [ f ]_\sim $ in certain quotient spaces $\ell^1(\h,\C)/ \mathord{\sim}$ of $\ell^1(\h,\C)$. The invertibility of $f$ in $\ell^1(\h,\C)$ (and hence in $\ell^1(\h,\R)$) is then equivalent to the invertibility of $ [ f ]_\sim $ in $\ell^1(\h,\C)/ \mathord{\sim}$ for each of these quotient spaces. We compare this approach with the cocycle method of \cite{LS2} described at the end of Section \ref{s:expansive}.

A second focus of this paper will be on homoclinic points of $(X_f, \alpha_f)$. A point $x \in X_f$ is \textit{homoclinic} if $\lim_{\gamma \to \infty} \alpha_f^\gamma x = 0$. A homoclinic point $x = (x_\gamma )\in X_f$ is \textit{summable} if $\sum_{\gamma \in \Gamma }\dT x_\gamma \dT <\infty $, where $\dT t\dT$ denotes the distance from $0$ of a point $t\in \mathbb{T}$. The summable homoclinic points play an essential role in proving specification and constructing symbolic covers of principal algebraic dynamical systems (\cite{LS1}). If $f \in \mathbb{Z}[\Gamma ]$ is invertible in $\ell ^1(\Gamma ,\mathbb{R})$ we obtain a summable homoclinic point in $X_f$ by reducing the coordinates of $f^{-1}\in\ell ^1(\Gamma ,\mathbb{R})$ $(\textup{mod}\;1)$ (cf. \cite{LS1}). However, if $f$ is not invertible, one may still be able to construct summable homoclinic points by using a multiplier method introduced in \cite{sv} for harmonic polynomials in $\mathbb{Z}[\mathbb{Z}^d]$ and subsequently studied for more general elements in $\mathbb{Z}[\mathbb{Z}^d]$ in \cite{lsv} and \cite{lsv2}. The method for this construction can be described as follows: if $f\in \mathbb{Z}[\Gamma ]$, and if there exists a fundamental solution $w\in \mathbb{R}^\Gamma $ of the equation $f\cdot w =1$ (where $1=1_{\mathbb{Z}[\Gamma ]}$ stands for the identity element in $\mathbb{Z}[\Gamma ]$) one can try to find a central element $g\in \Z[\Gamma ]$ such that the point $\phi =w \cdot g$ lies in $\ell^1(\h,\mathbb{R})$. By reducing the coordinates of $\phi $ $(\textup{mod}\;1)$ one again obtains a summable homoclinic point in $X_f$.

In the abelian case this method uses tools from commutative algebra and Fourier analysis. The nonabelian case under consideration here, where $\Gamma =\mathbb{H}$, requires some highly nontrivial combinatorial ideas, even for very elementary examples (like the element $f=2-x^{-1}-y^{-1}\in \mathbb{Z}[\mathbb{H}]$, expressed in terms of the usual generators $x,y,z$ of $\mathbb{H}$). It is worth mentioning that these are the first explicit examples of summable homoclinic points which have so far been found in the nonabelian and nonexpansive setting. With these summable homoclinic points we construct coding maps $\xi: \ell^\infty (\h,\Z) \longrightarrow (X_f,\alpha_f)$ and state a specification theorem. Furthermore, we show that the full $2$-shift on $\Gamma$, i.e.\ $\{0,1\}^\Gamma$, is an equal-entropy symbolic cover of $(X_f,\alpha_f)$ for the polynomial $f=2-x^{-1}-y^{-1} \in \Z[\h]$.

\subsection{Notation}\label{ss:notation} In these notes $\Gamma $ will always denote a countably infinite discrete group with identity element $1=1_\Gamma $, and $\Z [ \Gamma ]$ will stand for the integer group ring over $\Gamma $. We write $\ell ^\infty (\Gamma ,\mathbb{C})\subset \mathbb{C}^\Gamma $ for the space of bounded complex-valued maps $v=(v_\gamma )$ on $\Gamma $, where $v_\gamma $ is the value of $v$ at $\gamma $, and denote by $\|v\|_\infty =\sup_{\gamma \in\Gamma }|v_\gamma |$ the supremum norm on $\ell ^\infty (\Gamma ,\mathbb{C})$. For $1\le p <\infty $ we set 
$$\ell ^p(\Gamma ,\mathbb{C})=\{v=(v_\gamma )\in\ell ^\infty (\Gamma ,\mathbb{C}):\|v\|_p=\bigl(\sum_{\gamma \in\Gamma }|v_\gamma |^p\bigr)^{1/p}<\infty \}\,.$$ By $\ell ^p(\Gamma ,\mathbb{R})=\ell ^p(\Gamma ,\mathbb{\mathbb{C}})\cap \mathbb{\mathbb{R}}^\Gamma $ and $\ell ^p(\Gamma ,\mathbb{Z})=\ell ^p(\Gamma ,\mathbb{C})\cap \mathbb{Z}^\Gamma $ we denote the additive subgroups of real- and integer-valued elements of $\ell ^p(\Gamma ,\mathbb{C})$, respectively. If $1\le p <\infty $ and $q=\frac{p}{p-1}$ (with $q=\infty $ for $p=1$), $\ell ^q(\Gamma ,\mathbb{C})$ is the dual space of the Banach space $\ell ^p(\Gamma ,\mathbb{C})$. Hence every $w\in \ell ^q(\Gamma ,\mathbb{C})$ defines a bounded linear functional on $\ell ^p(\Gamma ,\mathbb{C})$, which we denote by
	\begin{equation}
	\label{eq:lplq}
v\mapsto (v,w)=\sum\nolimits_{\gamma \in \Gamma }v_\gamma \overline {w}_\gamma ,\enspace \enspace v\in \ell ^p(\Gamma ,\mathbb{C})
	\end{equation}
(where the bar denotes complex conjugation). For $1\le p<\infty $, $\ell ^p(\Gamma ,\mathbb{Z})=\ell ^1(\Gamma ,\mathbb{Z})$ is identified with $\mathbb{Z}[\Gamma ]$ by viewing each $g=\sum_{\gamma }g_\gamma \cdot \gamma \in\mathbb{Z}[\Gamma ]$ as the element $(g_\gamma )\in \ell ^1(\Gamma ,\mathbb{Z})$.

The group $\Gamma $ acts on $\ell ^p(\Gamma ,\mathbb{C})$ isometrically by left and right translations: for every $v\in\ell ^p(\Gamma ,\mathbb{C})$ and $\gamma \in\Gamma $ we denote by $\tilde{\lambda }^\gamma v$ and $\tilde{\rho }^\gamma v$ the elements of $\ell ^p(\Gamma ,\mathbb{C})$ satisfying $(\tilde{\lambda }^\gamma v)_{\gamma '}=v_{\gamma ^{-1}\gamma '}$ and $(\tilde{\rho }^\gamma v)_{\gamma '}=v_{\gamma '\gamma }$, respectively, for every $\gamma '\in\Gamma $. Note that $\tilde{\lambda }^{\gamma \gamma '}=\tilde{\lambda }^\gamma \circ \tilde{\lambda }^{\gamma '}$ and $\tilde{\rho }^{\gamma \gamma '}=\tilde{\rho }^\gamma \circ \tilde{\rho }^{\gamma '}$ for every $\gamma ,\gamma '\in\Gamma $.

The $\Gamma $-actions $\tilde{\lambda }$ and $\tilde{\rho }$ extend to actions of $\ell ^1(\Gamma ,\mathbb{C})$ on $\ell ^p(\Gamma ,\mathbb{C})$ which will again be denoted by $\tilde{\lambda }$ and $\tilde{\rho }$: for $h=(h_\gamma )\in \ell ^1(\Gamma ,\mathbb{C})$ and $v\in\ell ^p(\Gamma ,\mathbb{C})$ we set
	\begin{displaymath}
\tilde{\lambda }^hv=\sum\nolimits_{\gamma \in \Gamma }h_\gamma \tilde{\lambda }^\gamma v,\qquad \tilde{\rho }^hv=\sum\nolimits_{\gamma \in \Gamma }h_\gamma \tilde{\rho }^\gamma v.
	\end{displaymath}
These expressions correspond to the usual convolutions
	\begin{displaymath}
\tilde{\lambda }^hv=h*v,\qquad \tilde{\rho }^hv=v*\bar{h}^*,
	\end{displaymath}
where $h\mapsto h^*$ is the involution on $\ell ^1(\Gamma ,\mathbb{C})$ given by $h^*_\gamma =\bar{h}_{\gamma ^{-1}},\,\gamma \in \Gamma $.

If $\mathcal{H}$ is a (complex) Hilbert space with inner product $\langle \cdot ,\cdot \rangle $ we denote by $\mathcal{B}(\mathcal{H})$ the algebra of bounded linear operators on $\mathcal{H}$ and write $\mathcal{U}(\mathcal{H})\subset \mathcal{B}(\mathcal{H})$ for the group of unitary operators on $\mathcal{H}$, furnished with the strong (or, equivalently, weak) operator topology. Every unitary representation $\pi \colon \Gamma \longrightarrow \mathcal{U}(\mathcal{H})$ of $\Gamma $ can be extended to a representation $\tilde{\pi }\colon \ell ^1(\Gamma ,\mathbb{C})\longrightarrow \mathcal{B}(\mathcal{H})$ of the algebra $\ell ^1(\Gamma ,\mathbb{C})$ by setting $\tilde{\pi }(h)=\sum_{\gamma \in \Gamma }h_\gamma \pi (\gamma )$ for every $h=(h_\gamma )\in \ell ^1(\Gamma ,\mathbb{C})$. This map is continuous (w.r.t. the strong operator topology on $\mathcal{B}(\mathcal{H})$) and satisfies that $\tilde{\pi }(h)^*= \tilde{\pi }(h^*)$ for all $h\in \ell ^1(\Gamma ,\mathbb{C})$, where $\tilde{\pi }(h)^*$ is the adjoint operator of $\tilde{\pi }(h)$. The representation $\pi$ of $\Gamma$ (or, equivalently, the representation $\tilde{\pi }$ of $\ell ^1(\Gamma ,\mathbb{C})$) is \textit{irreducible} if it does not admit a nontrivial invariant subspace of $\mathcal{H}$. Clearly, if $\pi $ is irreducible, it maps every central element of $\Gamma $ to a scalar multiple of the identity operator on $\mathcal{H}$.

\subsection{States and representations}\label{ss:states} A linear map $\phi \colon \ell ^1(\Gamma ,\mathbb{C})\longrightarrow \mathbb{C}$ is a \textit{state} on $\ell ^1(\Gamma,\C)$ if $\phi (1_{\ell ^1(\Gamma ,\mathbb{C})})=1$ and $\phi (h^*h)\ge0$ for every $h\in \ell ^1(\Gamma ,\mathbb{C})$. If $\pi $ is a unitary representation of $\Gamma$ on a Hilbert space $\mathcal{H}$ and $v\in \mathcal{H}$ is a unit vector, then the map
	\begin{equation}
	\label{eq:state}
h\mapsto \phi (h)=\langle \tilde{\pi }(h)v,v\rangle,\quad h\in \ell ^1(\Gamma ,\mathbb{C}),
	\end{equation}
is a state on $\ell ^1(\Gamma ,\mathbb{C})$. Moreover, every state on $\ell ^1(\Gamma ,\mathbb{C})$ arises in this way: if $\phi $ is a state on $\ell ^1(\Gamma ,\mathbb{C})$, then there exists a cyclic\footnote{A unitary representation $\pi $ of $\Gamma $ on a Hilbert space $\mathcal{H}$ is \textit{cyclic} if there exists a vector $v\in \mathcal{H}$ whose orbit $\{\pi (\gamma )v:\gamma \in \Gamma \}$ spans $\mathcal{H}$; such a vector $v$ is called \textit{cyclic}.} unitary representation $\pi _\phi $ of $\Gamma $ on a complex Hilbert space $\mathcal{H}_\phi $ with cyclic vector $v_\phi $ satisfying \eqref{eq:state}. The state $\phi $ determines the pair $(\pi _\phi ,v_\phi )$ uniquely up to unitary equivalence: if $\pi '$ is a second representation of $\Gamma $ on a Hilbert space $\mathcal{H}'$, and if $v'\in \mathcal{H}'$ is a unit vector satisfying \eqref{eq:state}, then there exists a unitary operator $W\colon \mathcal{H}_\phi \longrightarrow \mathcal{H}'$ with $Wv_\phi =v'$ and $W\circ \pi _\phi (\gamma )=\pi '(\gamma )\circ W$ for every $\Gamma \in \Gamma $. A state $\phi $ on $\ell ^1(\Gamma ,\mathbb{C})$ is \textit{pure} if the corresponding representation $\pi _\phi $ of $\Gamma $ (or, equivalently, the representation $\tilde{\pi }_\phi $ of $\ell ^1(\Gamma ,\mathbb{C})$) is irreducible.

Since every bounded linear functional $\phi \colon \ell ^1(\Gamma ,\mathbb{C})\longrightarrow \mathbb{C}$ is of the form \eqref{eq:lplq} for some $w\in \ell ^\infty (\Gamma ,\mathbb{C})$ there exists, for every state $\phi $ on $\ell ^1(\Gamma ,\mathbb{C})$, an element $w\in \ell ^\infty (\Gamma ,\mathbb{C})$ satisfying
	\begin{equation}
	\label{eq:state2}
\phi (h) = (h,w) = \sum\nolimits_{\gamma \in \Gamma }h_\gamma \overline{w}_\gamma
	\end{equation}
for every $h\in \ell ^1(\Gamma ,\mathbb{C})$. In particular,
	\begin{equation}
	\label{eq:state3}
w_{1_\Gamma }=1\enspace \enspace \textup{and}\enspace \enspace (h^*h,w)\ge0\enspace \enspace \textup{for every}\;h\in \ell ^1(\Gamma ,\mathbb{C}).
	\end{equation}
We denote by $\mathcal{P}(\Gamma )\subset \ell ^\infty (\Gamma ,\mathbb{C})$ the set of all elements $w\in \ell ^\infty (\Gamma ,\mathbb{C})$ satisfying \eqref{eq:state3}. The elements of $\mathcal{P}(\Gamma )$ are called (\textit{normalised}) \textit{positive definite functions} on $\Gamma $.

\subsection{The Discrete Heisenberg group}\label{ss:heisenberg} A canonical generating set of the discrete Hei\-senberg group $\mathbb{H}$ in \eqref{eq:Heisenberg} is given by $S=\{x,x^{-1},y,y^{-1}\}$, where
	$$
x = \left(
	\begin{smallmatrix}
1 & 1 & 0
	\\
0 & 1 & 0
	\\
0 & 0 & 1
	\end{smallmatrix}
\right), \quad y = \left(
	\begin{smallmatrix}
1 & 0 & 0
	\\
0 & 1 & 1
	\\
0 & 0 & 1
	\end{smallmatrix}
\right) .
	$$
The center of $\h$ is generated by
	$$
\smash[t]{z= xyx^{-1}y^{-1}= \left(
	\begin{smallmatrix}
1 & 0 & 1
	\\
0 & 1 & 0
	\\
0 & 0 & 1
	\end{smallmatrix}
\right).}
	$$
The generators $x,y,z$ satisfy the relations
	\begin{equation}
	\label{eq:relations}
xz=zx,\enspace yz=zy,\enspace x^ky^l=y^lx^kz^{kl},\enspace k,l\in\mathbb{Z}.
	\end{equation}

\subsection{$q$-binomial coefficients} The \textit{q-binomial coefficients} are defined by
	\begin{equation}
	\label{eq:q-binomial}
\qbinom{n}{k}= \prod_{i=0}^{k-1} \frac{1-q^{n-i}}{1-q^{i+1}} \quad \textit{for all $n \in \N$ and $0 \leq k \leq n$}.
	\end{equation}
The connection between $q$-binomial coefficients and the discrete Heisenberg group is explained by the following fact: if the generators $x,y \in \h$ fulfil the commutation relation $xy=zyx$, then
	$$
(x+y)^n =\sum_{k=0}^n \qbinom{n}{k} x^k y^{n-k}
	$$
where $\qbinom{n}{k}$ are $q$-binomial coefficients with $q=z^{-1}$. We will refer to this fact as the {\it $q$-binomial theorem}.

The $q$-binomial coefficients $\qbinom{n}{k}$ are polynomials of degree $k(n-k)$ with nonnegative integer coefficients. Moreover, they are \textit{unimodal}, i.e.\ if $\qbinom{n}{k}=\sum_{i=0}^{k(n-k)} c_i q^i$, then for $r=\lceil k(n-k)/2 \rceil $ one has
	$$
c_0 \leq c_1 \leq c_2 \leq \cdots \leq c_r \geq c_{r+1}\geq \cdots \geq c_{k(n-k)-1}\geq c_{k(n-k)}.
	$$
The $q$-binomial coefficients are \textit{symmetric} in the sense that $\qbinom{n}{k} =\qbinom{n}{n-k}$.

\section{Expansive algebraic actions}
	\label{s:expansive}

We recall that an algebraic action $\alpha $ of a countably infinite discrete group $\Gamma $ on a compact group $X$ is \textit{expansive} if it satisfies \eqref{eq:expansive} or, equivalently, if there exists a neighbourhood $U$ of the identity element $0_X\in X$ with
	\begin{equation}
	\label{eq:expansive2}
\bigcap_{\gamma\in\Gamma}\alpha^{\gamma}(U)=\{0_X\}.
	\end{equation}

In this section we discuss characterisations of expansiveness of principal algebraic actions of a countably infinite discrete group $\Gamma $.

	\begin{theo}
	\label{t:expa}
Let $\Gamma $ be a countably infinite discrete group, $f\in \mathbb{Z}[\Gamma ]$, and let $(X_f,\alpha _f)$ be the principal algebraic $\Gamma $-action defined by $f$ \textup{(}Definition \ref{d:principal}\textup{)}. The following statements are equivalent:
	\begin{enumerate}
	\item
The principal algebraic $\Gamma $-action $\alpha _f$ is expansive;
	\item
$f$ is invertible in $\ell^1 (\Gamma,\R)$;
	\item
$K_\infty (f)\coloneqq \{v\in \ell^\infty (\Gamma,\C)\,:\, \tilde{\rho }^fv = v*f^* = 0\} = \{0\}$;
	\item
The one-sided ideals $\mathbb{Z}[\Gamma ]f\subset \mathbb{Z}[\Gamma ]$ and $f\mathbb{Z}[\Gamma ]\subset \mathbb{Z}[\Gamma ]$ both contain lopsided\,\footnote{An element $g=\sum_{\gamma }g_\gamma \cdot \gamma \in \mathbb{Z}[\Gamma ]$ is \textit{lopsided} if there exists a $\gamma _0\in \Gamma $ with $|g_{\gamma _0}|>\sum_{\gamma \in \Gamma \smallsetminus \{\gamma _0\}}|g_\gamma |$. In this case we call $g_{\gamma _0}$ the \textit{dominant} coefficient of $g$.} elements.
	\end{enumerate}
	\end{theo}

	\begin{proof}
The equivalence of (1), (2) and (3) was shown in \cite[Theorem 3.2]{DS}. The equivalence of (2) and (4) was pointed out to us by Hanfeng Li; we are grateful to him for permitting us to include the following argument.

If an element $g\in \mathbb{Z}[\Gamma ]$ is lopsided with dominant coefficient $g_{\gamma _0}$ we assume for simplicity that $g_{\gamma _0}>0$ and set $h=\gamma _0^{-1}g$, where $\textup{sgn}(g_{\gamma _0})$ is the sign of $g_{\gamma _0}$. Then $h_1$ is the dominant coefficient of $h$, $h_1=|g_{\gamma _0}|>0$, and $h'=\frac{1}{h_1}\cdot h-1\in \ell ^1(\Gamma ,\mathbb{R})$ satisfies that $\|h'\|_1<1$ and $h=h_1\cdot (1+h')$. Hence
	\begin{displaymath}
h^{-1}=\tfrac{1}{h_1}(1+h')^{-1}=\tfrac{1}{h_1}\cdot \bigl(\textstyle\sum_{n\ge0}(-1)^n\cdot {h'}^n\bigr) \in \ell ^1(\Gamma ,\mathbb{R}),
	\end{displaymath}
which implies that $g^{-1}=h^{-1}\cdot \gamma _0\in \ell ^1(\Gamma ,\mathbb{R})$.

If $\mathbb{Z}[\Gamma ]f$ contains a lopsided element $hf$, then the preceding paragraph shows that there exists a $v\in \ell ^1(\Gamma ,\mathbb{R})$ with $v*(hf)=(v*h)*f = 1$. Hence $f$ has a left inverse in $\ell ^1(\Gamma ,\mathbb{R})$, which implies that $f$ is invertible in $\ell ^1(\Gamma ,\mathbb{R})$ (cf. \cite[p. 122]{K}). Similarly we see that $f$ is invertible in $\ell ^1(\Gamma ,\mathbb{R})$ whenever the right ideal $f\mathbb{Z}[\Gamma ]$ (or, equivalently, the left ideal $\mathbb{Z}[\Gamma ]f^*$) contains a lopsided element.

Now assume that $f$ is invertible in $\ell ^1(\Gamma ,\mathbb{R})$ and set $v=f^{-1}$. Let $q>\|f\|_1$ be an integer and choose an element $w\in \ell ^1(\Gamma ,\mathbb{R})$ with rational coordinates $w_\gamma =\frac{k_\gamma }{q}\,,\,k_\gamma \in\mathbb{Z}$, such that $\| v -w\|_1 \le\frac{1}{2q}$, hence $w\in\Q [\Gamma]$. Then $h\coloneqq qw\in \mathbb{Z}[\Gamma ]$ and
	\begin{displaymath}
\textstyle{|(hf)_{1_{\Gamma}} -q(v*f)_{1_{\Gamma}} |=q\cdot |(wf)_{1_{\Gamma}} -(v*f)_{1_{\Gamma}} | \le q\cdot \frac{1}{2q} \cdot \|f\|_1 = \frac{\|f\|_1}{2}} \,.
	\end{displaymath}
 Since
	\begin{displaymath}
\smash[t]{q(v*f)_\gamma =\begin{cases}q >\|f\|_1&\textup{if}\;\gamma =1_\Gamma ,\\0&\textup{otherwise}, \end{cases}}
	\end{displaymath}
	we obtain that
	$$
	\textstyle{|(hf)_{1_{\Gamma}}| > \frac{\|f\|_1}{2}}\,.
	$$
	Moreover,
	\begin{align*}
	\sum_{\gamma \in \Gamma \setminus \{ e_{\Gamma}\} } \textstyle{|(hf)_\gamma |} & =
	\sum_{\gamma\in\Gamma \setminus \{ e_{\Gamma}\}} \textstyle{|(hf)_\gamma -q(v*f)_{\gamma} |} 
	\leq \displaystyle{\sum_{\gamma \in \Gamma}} \textstyle{|(hf)_\gamma -q(v*f)_{\gamma} |} 
	 \\&=\sum_{\gamma\in\Gamma} q\cdot \textstyle{|(wf)_\gamma -(v*f)_\gamma |}
\leq q \cdot \|w-v\|_1 \cdot \|f\|_1  \le  q\cdot \frac{1}{2q} \cdot \|f\|_1 = \frac{\|f\|_1}{2} \,.
	\end{align*}
Hence, the element $hf\in \mathbb{Z}[\Gamma ]f$ is lopsided.

Since the invertibility of $f$ in $\ell ^1(\Gamma ,\mathbb{R})$ is equivalent to that of $f^*$, the preceding paragraph shows that the left ideal $\mathbb{Z}[\Gamma ]f^*$ contains a lopsided element $hf^*$. Then $(hf^*)^*=fh^*\in f\mathbb{Z}[\Gamma ]$ is again lopsided. This completes the proof of the equivalence of the conditions (2) and (4).
	\end{proof}

We list a few elementary consequences of Theorem \ref{t:expa} (cf. also \cite{DS} and \cite{ER}).

	\begin{coro}
	\label{c:expa}
Let $\Gamma $ be a countably infinite discrete group, $f \in \Z[\Gamma]$, and let $(X_f,\alpha _f)$ be the associated principal algebraic $\Gamma $-action.
	\begin{enumerate}
	\item
The following conditions are equivalent.
	\begin{enumerate}
	\item
$\alpha _f$ is expansive;
	\item
$\alpha _{f^*}$ is expansive;
	\item
The ideal $\mathbb{R}\otimes (\mathbb{Z}[\Gamma ]f)$ is dense in $\ell^1(\Gamma,\mathbb{R})$;
	\item
$K_\infty (f)=\{v\in \ell^{\infty}(\Gamma,\C): (g,v)=0 \, \,\text{for all}\,\, g \in \mathbb{Z}[\Gamma ]f \} = \{0\}$;
	\end{enumerate}
	\item
If there exists a $w\in \ell ^\infty (\Gamma ,\mathbb{R})$ such that $\tilde{\rho }^fw=w*f^*=1_{\mathbb{Z}[\Gamma ]}$, then $\alpha _f$ is expansive if and only if $w\in \ell ^1(\Gamma ,\mathbb{R})$;
	\end{enumerate}
	\end{coro}

	\begin{proof}
The equivalence of (a) and (b) in (1) follows from Theorem \ref{t:expa} and the fact that $f$ is invertible in $\ell ^1(\Gamma ,\mathbb{R})$ if and only if the same is true for $f^*$. The conditions (1.c) and (1.d) are equivalent by the Hahn-Banach Theorem, and both are equivalent to expansiveness by Theorem \ref{t:expa}.

If $\alpha _f$ is expansive, then both $f$ and $f^*$ are invertible in $\ell ^1(\Gamma ,\mathbb{R})$ and $K_\infty (f)\coloneqq \{v\in \ell^\infty (\Gamma,\C)\,:\, \tilde{\rho }^fv = v*f^* = 0\} = \{0\}$ by Theorem \ref{t:expa}. The equation $\tilde{\rho }^fw=w*f^*=1_{\mathbb{Z}[\Gamma ]}$ thus has a unique solution $w\in \ell ^\infty (\Gamma ,\mathbb{R})$. Since $v=(f^*)^{-1}$ is also a solution of this equation it follows that $w=(f^*)^{-1}\in \ell ^1(\Gamma ,\mathbb{R})$. The reverse implication in (2) is also clear from Theorem \ref{t:expa}.
	\end{proof}

If the group $\Gamma $ is nilpotent, the sufficient condition (1.d) of expansiveness in Corollary \ref{c:expa} can be weakened.

	\begin{theo}[\protect{\cite[Theorem 8.2]{ER}}]
	\label{t:ER}
If $\Gamma $ is a countably infinite nilpotent group and $f \in \Z[\Gamma]$, then the principal algebraic action $\alpha_f$ is nonexpansive if and only if there exists an irreducible unitary representation $\pi $ of $\Gamma $ on a Hilbert space $\mathcal{H}$ and a unit vector $v \in \mathcal{H}$ such that $\pi(f) v =0$.
	\end{theo}

Since the proof of Theorem \ref{t:ER} was omitted in \cite{ER}, we include it for completeness.

	\begin{proof}
Since $\Gamma $ is nilpotent, the group algebra $\ell ^1(\Gamma ,\mathbb{C})$ is \textit{symmetric} in the sense that $1+g^*g$ is invertible in $\ell ^1(\Gamma ,\mathbb{C})$ for every $g\in \ell ^1(\Gamma ,\mathbb{C})$ and hence for every $g\in \mathbb{Z}[\Gamma ]$ (cf. \cite[Corollary to Theorem 2]{Hulanicki}). According to \cite[\S 23.3, p. 313]{Neumark} there exists, for every closed left ideal $I\subset \ell ^1(\Gamma ,\mathbb{C})$, a normalised positive definite function $p\in \ell ^\infty (\Gamma ,\mathbb{C})$ with $(h,p)=0$ for every $h\in I$ (cf. Subsection \ref{ss:states}).

If $\alpha _f$ is nonexpansive, Corollary \ref{c:expa} yields a nonzero element $w\in \ell ^\infty (\Gamma ,\mathbb{C})$ with $(gf,w)=0$ for every $g\in \mathbb{Z}[\Gamma ]$, and hence for every $g\in \ell ^1(\Gamma ,\mathbb{C})$. The closure $I=\overline{\ell ^1(\Gamma ,\mathbb{C})f}$ is thus a closed left ideal in $\ell ^1(\Gamma ,\mathbb{C})$, and the preceding paragraph shows that there exists a normalised positive definite function $p\in \mathcal{P}(\Gamma )$ with $\phi (hf)\coloneqq (hf,p)=0$ for every $h\in \mathbb{Z}[\Gamma ]$. If $\pi _\phi $ is a cyclic unitary representation of $\Gamma $ on a Hilbert space $\mathcal{H}_\phi $ and $v_\phi \in \mathcal{H}$ a cyclic unit vector satisfying $\phi (h)=(h,p)=\langle \tilde{\pi }_\phi (h)v_\phi ,v_\phi \rangle $ for every $h\in \ell ^1(\Gamma ,\mathbb{C})$ (cf. \eqref{eq:state} -- \eqref{eq:state3}), then we obtain that
	\begin{displaymath}
\langle \tilde{\pi }_\phi (f)v_\phi ,\tilde{\pi }_\phi (g)v_\phi \rangle =\langle \tilde{\pi }_\phi (g)^*\tilde{\pi }_\phi (f) v_\phi ,v_\phi \rangle = \langle \tilde{\pi }_\phi (g^*f) v_\phi ,v_\phi \rangle = 0
	\end{displaymath}
for every $g\in \mathbb{Z}[\Gamma]$. Since $v_\phi $ is cyclic for $\pi _\phi $ it follows that $\tilde{\pi }_\phi (f)v_\phi =0$. By decomposing the representation $\pi _\phi $ into irreducibles we have proved that nonexpansiveness of $\alpha _f$ implies the existence of an irreducible unitary representation $\pi $ of $\Gamma $ on a Hilbert space $\mathcal{H}$ and of a unit vector $v\in \mathcal{H}$ with $\tilde{\pi }(f)v=0$.

The reverse implication is obvious: if there exists a unitary representation $\pi $ of $\Gamma $ on a Hilbert space $\mathcal{H}$ and a unit vector $v\in \mathcal{H}$ with $\tilde{\pi }(f)v=0$, then the equations \eqref{eq:state} -- \eqref{eq:state3} define a normalised positive definite element $w\in \ell ^\infty (\Gamma ,\mathbb{C})$ with $(g,w)=0$ for every $g\in \mathbb{Z}[\Gamma ]f$ (cf. Corollary \ref{c:expa} (1)).
	\end{proof}

We illustrate Theorem \ref{t:ER} with two examples.

	\begin{exam}[\protect{\cite[Theorem 6.5]{Schmidt}}]
	\label{e:Zd}
Put $\Gamma =\mathbb{Z}^d$ and identify $\mathbb{Z}[\mathbb{Z}^d]$ with the ring $R_d=\mathbb{Z}[u_1^{\pm1},\dots ,u_d^{\pm1}]$ of Laurent polynomials in $d$ variables with integral coefficients by viewing every $f = \sum_{\mathbf{n}\in\mathbb{Z}^d}f_\mathbf{n}\cdot \mathbf{n}\in \mathbb{Z}[\mathbb{Z}^d]$ as the Laurent polynomial $\sum_{\mathbf{n}\in\mathbb{Z}^d}f_\mathbf{n} u^\mathbf{n}\in R_d$ with $u^\mathbf{n}=u_1^{n_1}\cdots u_d^{n_d}$ for every $\mathbf{n}=(n_1,\dots ,n_d)$. Denote by
	\begin{gather*}
\mathsf{V}(f)=\{(z_1,\dots,z_d)\in (\C^\times )^d :f(z_1,\dots,z_d)=0\},
	\\
\mathsf{U}(f) =\{(z_1,\dots,z_d)\in\text{V}(f):|z_1|=\dots=|z_d|=1\}.
	\end{gather*}
the \textit{complex} and \textit{unitary varieties} of $f$. Then $\alpha_f$ is expansive if and only if $\mathsf{U}(f)=\varnothing $. Note that this condition is equivalent to saying that there is no irreducible (and hence one-dimensional) unitary representation $\pi $ of $\mathbb{Z}^d$ on the Hilbert space $\mathcal{H}=\mathbb{C}$ satisfying that $\tilde{\pi }(f)1=0$.
	\end{exam}

	\begin{exam}[\protect{\cite[Lemma 8.3]{ER}}]
	\label{e:ER1}
Let $\Gamma $ be nilpotent, and let $f\in \Z[\Gamma ]$ have the properties that $0<f_{1_\Gamma }= \sum_{\gamma \in \Gamma \smallsetminus \{1_\Gamma \}}|f_\gamma|$, and that $\textup{supp}(f)=\{\gamma \in \Gamma :f_\gamma \ne 0\}$ generates $\Gamma $. \textit{Then $\alpha _f$ is nonexpansive if and only if there exists a homomorphism $\chi \colon \Gamma \longrightarrow \mathbb{S}=\{c\in \mathbb{C}:|c|=1\}$ with $\chi (\gamma )=-\textup{sgn}(f_\gamma )$ for every $\gamma $ in $\textup{supp}(f)\smallsetminus \{1_{\h}\}$, where $\textup{sgn}(f_\gamma )$ is the sign of $f_\gamma $.}

Indeed, if $\alpha _f$ is nonexpansive, then Theorem \ref{t:ER} implies the existence of a unitary representation $\pi $ of $\Gamma $ on a Hilbert space $\mathcal{H}$ and of a unit vector $v\in \mathcal{H}$ with $\tilde{\pi }(f)v=0$. In view of the special form of $f$ this means that $v=\sum_{\gamma \in \Gamma \smallsetminus \{1_\Gamma \}}-(f_\gamma /f_{1_\Gamma })\pi (\gamma )v$. Since the unit ball in a Hilbert space is strictly convex, it follows that $\pi (\gamma )v=-\textup{sgn}(f_\gamma )v$ for every $\gamma \in \textup{supp}(f)\smallsetminus \{1_{\h}\}$.

As the one-dimensional subspace $V=\{tv:t\in\mathbb{C}\}\subset \mathcal{H}$ is invariant under $\pi $, there exists a group homomorphism $\chi \colon \Gamma \longrightarrow \mathbb{S}$ with $\pi (\gamma )v=\chi (\gamma )v$ for every $\gamma \in \Gamma $ and $\tilde{\chi }(f)\coloneqq \sum_{\gamma \in \Gamma }f_\gamma \chi (\gamma )=0$. The latter condition is equivalent to saying that $\chi (\gamma )=-\textup{sgn}(f_\gamma )$ for every $\gamma \in \textup{supp}(f)\smallsetminus \{1_{\h}\}$.

Conversely, if there exists a homomorphism $\chi \colon \Gamma \longrightarrow \mathbb{S}$ with $\tilde{\chi }(f)=0$, then $\alpha _f$ is nonexpansive by Theorem \ref{t:ER}.

Note that every homomorphism $\chi \colon \Gamma \longrightarrow \mathbb{S}$ satisfies that $\chi \equiv 1$ on $[\Gamma ,\Gamma ]$, the commutator subgroup of $\Gamma $. If $\Gamma =\h$, the discrete Heisenberg group in Subsection \ref{ss:heisenberg} with the generators $x,y,z$ defined there, then $[\Gamma ,\Gamma ]=\{z^n:n\in \mathbb{Z}\}$.

To illustrate this result we consider any polynomial $f\in \mathbb{Z}[\h]$ of the form $f=|a|+|b|+|c| + a\cdot x + b\cdot y +c\cdot z=(|a|+|b|+|c|)\cdot 1_{\h} + a\cdot x + b\cdot y +c\cdot z$ with $ab\ne0$. Then $\alpha _f$ is expansive and if only if $c>0$ (\cite[Example 8.4]{ER}).
	\end{exam}

\subsection{Expansiveness of actions of the discrete Heisenberg group}\label{ss:cocycle}

In \cite{LS2} the question was raised whether there is an effective characterisation of expansiveness for principal algebraic actions of a nonabelian group $\Gamma $. A natural first step in trying to answer this question is to look at nilpotent groups, with the discrete Heisenberg group $\h$ in Subsection \ref{ss:heisenberg} as the primary example.

Assume therefore that $\Gamma =\h$, and denote by $x,y,z\in \h$ the three noncommuting variables defined in Subsection \ref{ss:heisenberg}. We write a typical element $x^{k_1}y^{k_2}z^{k_3}\in \h$ as $[k_1,k_2,k_3]$. Then the $\h$-actions $\tilde{\lambda }$ and $\tilde{\rho }$ on $\ell ^\infty (\h,\mathbb{C})$ in Subsection \ref{ss:notation} take the form
	\begin{equation}
	\label{eq:actionH}
	\begin{gathered}
(\tilde{\lambda }^{[m_1,m_2,m_3]}w)_{[n_1,n_2,n_3]}=w_{[n_1-m_1,n_2-m_2,n_3-m_3-m_1m_2+m_2n_1]},
	\\
(\tilde{\rho }^{[m_1,m_2,m_3]}w)_{[n_1,n_2,n_3]}= w_{[n_1+m_1,n_2+m_2,n_3+m_3-m_1n_2]}
	\end{gathered}
	\end{equation}
for every $w\in\ell ^\infty (\h,\mathbb{C})$ and $[m_1,m_2,m_3],[n_1,n_2,n_3]\in\h$, and every $f\in \mathbb{Z}[\h]$ is written as
	\begin{equation}
	\label{eq:f}
f=\sum\nolimits_{(k_1,k_2,k_3)\in\mathbb{Z}^3} f_{(k_1,k_2,k_3)}[k_1,k_2,k_3]
	\end{equation}
with $f_{(k_1,k_2,k_3)}=0$ for all but finitely many $(k_1,k_2,k_3)\in \mathbb{Z}^3$.

A `cocycle' method for analysing the expansiveness of $\alpha _f$ for certain elements $f \in \Z[\h]$ was proposed in \cite{LS2}. We mention this method, since it provides us in some special cases with a verification of the approach developed in the next section.

Although the cocycle approach from \cite{LS2} works for much more general $f \in \Z[\h]$, it seems to be particularly useful for elements $f \in \Z [\h]$ which are `linear' in either of the variables $x$ or in $y$.

To be more specific, assume that $f\in\mathbb{Z}[\h]$ is of the form
	\begin{equation}
	\label{eq:form1}
f=g_1y+g_0
	\end{equation}
with $g_i\in \mathbb{Z}[x^{\pm1},z^{\pm1}]\cong R_2$ and
	\begin{equation}
	\label{eq:form2}
\mathsf{U}(g_0)=\mathsf{U}(g_1)=\varnothing .
	\end{equation}
Condition \eqref{eq:form2} is equivalent to assuming that the principal algebraic $\mathbb{Z}^2$-actions defined by the Laurent polynomials $g_0$ and $g_1$ are expansive. Under these assumptions one has the following characterisation of expansiveness for $\alpha _f$.

	\begin{theo}[\cite{LS2}]
	\label{t:LS2}
Let $f\in \Z[\h]$ be of the form \eqref{eq:form1} -- \eqref{eq:form2}. For every $\theta \in \mathbb{S}$ we define a continuous map $\phi _\theta \colon \mathbb{S}\longrightarrow \mathbb{R}$ by
	\begin{equation}
	\label{eq:phi}
\phi _\theta (\xi )=\log \bigl|\tfrac {g_0(\xi , \theta )}{ g_1(\xi , \theta )}\bigr|.\vspace{1mm}
	\end{equation}
Then $\alpha _f$ is nonexpansive if and only if at least one of the following conditions is satisfied.
	\begin{enumerate}
	\item
There exists a $\theta \in \mathbb{S}$ such that $\int \phi _\theta \,d\lambda =0$, where $\lambda $ is the normalised Lebesgue measure on $\mathbb{S}$;
	\item
There exist a $p\ge1$ and a $p$-th root of unity $\theta \in \mathbb{S}$ such that $\sum_{j=0}^{p-1}\phi _\theta (\xi \theta ^j)=0$ for some $\xi \in \mathbb{S}$.
	\end{enumerate}
	\end{theo}

	\begin{rema}
	\label{r:LS2}
Theorem \ref{t:LS2} has the following equivalent formulation. If $f\in \Z[\h]$ is of the form \eqref{eq:form1} -- \eqref{eq:form2}, then $\alpha _f$ is nonexpansive if and only if there exists a $\theta \in \mathbb{S}$ such that $\int \phi _\theta \,d\nu =0$ for some probability measure $\nu $ on $\mathbb{S}$ which is invariant and ergodic under the rotation $R_\theta \colon \xi \mapsto \xi \theta ,\;\xi \in \mathbb{S}$.
	\end{rema}

	\begin{coro}
	\label{c:LS2}
Let $f\in \Z[\h]$ be of the form \eqref{eq:form1} -- \eqref{eq:form2}. If
	\begin{displaymath}
\iint \bigl(\log |g_0(\xi , \theta )| - \log |g_1(\xi , \theta )|\bigr)\,d\lambda (\xi )\,d\lambda (\theta )=0
	\end{displaymath}
then $\alpha _f$ is nonexpansive.
	\end{coro}

	\begin{coro}
	\label{c:LS3}
Let $\tilde{f}\in\mathbb{Z}[\h]$ be of the form $\tilde{f}=x\tilde{g}_1+\tilde{g}_0$ with $\tilde{g}_i\in \mathbb{Z}[y^{\pm1},z^{\pm1}]\cong R_2$ and $\mathsf{U}(\tilde{g}_0)=\mathsf{U}(\tilde{g}_1)=\varnothing $. If $\tilde{\phi }_\theta \colon \mathbb{S}\longrightarrow \mathbb{R}$ is defined by $\tilde{\phi }_\theta (\eta )=\log \bigl|\tfrac {\tilde{g}_0(\eta , \theta )}{\tilde{g}_1(\eta , \theta )}\bigr|$ for every $\theta ,\eta \in\mathbb{S}$, then $\alpha _{\tilde{f}}$ is nonexpansive if and only if $\tilde{\phi }_\theta $ satisfies the condition in Remark \ref{r:LS2}.
	\end{coro}

Even if $f$ does not satisfy \eqref{eq:form2}, the method of Theorem \ref{t:LS2} may still be useful for proving nonexpansiveness of $\alpha _f$. We have the following corollary of the proof of Theorem \ref{t:LS2}.

	\begin{coro}
	\label{c:LS4}
Assume that $f\in\mathbb{Z[\h]}$ is of the form \eqref{eq:form1}, and that there exist a $\theta \in\mathbb{S}$ and a probability measure $\nu $ on $\mathbb{S}$ which is invariant and ergodic under the rotation $R_\theta \colon \mathbb{S}\longrightarrow \mathbb{S}$, such that the following conditions hold:
	\begin{enumerate}
	\item
$g_0(\xi ,\theta )g_1(\xi ,\theta )\ne0$ for every $\xi \in \textup{supp}(\nu )$ \textup{(}the support of $\nu $\textup{)};
	\item
$\int \phi _\theta \,d\nu =0$, where $\phi _\theta $ is given by \eqref{eq:phi}.
	\end{enumerate}
Then $\alpha _f$ is nonexpansive.

\smallskip The analogous statement holds in the setting of Corollary \ref{c:LS3}.
	\end{coro}

The proofs of Theorem \ref{t:LS2} and Corollary \ref{c:LS3}, taken from \cite{LS2}, are included for convenience of the readers.

	\begin{proof}[Proof of Theorem \ref{t:LS2}]
For every $\theta \in \mathbb{S}$ we define $\psi _\theta \colon \mathbb{S}\longrightarrow \mathbb{R}$ by
	\begin{displaymath}
\psi _\theta (\xi )=\log \bigl|\tfrac {g_0(\xi , \theta )}{ g_1(\xi \theta ^{-1}, \theta )}\bigr|
	\end{displaymath}
and consider the map $c_\theta \colon \mathbb{Z}\times \mathbb{S}\longrightarrow \mathbb{R}$ given by
	\begin{equation}
	\label{cocycle}
c_\theta (n, \xi ) =
	\begin{cases}
\sum_{k=0}^{n-1}\psi _\theta (\xi \theta ^{-k})& \textup{for}\enspace n \ge 1,
	\\
0 & \textup{for} \enspace n=0
	\\
-c_\theta (-n,\xi \theta ^{-n})&\textup{for}\enspace n<0,
	\end{cases}
	\end{equation}
which satisfies the cocycle equation
	\begin{equation}
	\label{eq:cocycle2}
c_\theta (m,\xi \theta ^{-n})+ c_\theta (n,\xi )= c_\theta (m+n,\xi )
	\end{equation}
for all $m,n\in \mathbb{Z}$ and $\xi \in \mathbb{S}$.

Suppose that $\alpha _f$ is nonexpansive. By Theorem \ref{t:expa} (3), the $\tilde{\lambda }$-invariant subspace $K_\infty (f)\linebreak[0]=\{v\in \ell ^\infty (\h ,\mathbb{C}):\tilde{\rho }^fv=0\}$ is nonzero. By restricting $\tilde{\lambda }$ to the abelian subgroup $\Delta =\{x^kz^l:(k,l)\in \mathbb{Z}^2\}\subset \h $ we can apply the argument in \cite[Lemma 6.8]{Schmidt} to see that $K_\infty (f)$ contains a one-dimensional subspace $V$ which is invariant under the restriction of $\tilde{\lambda }$ to $\Delta $. In other words, there exist a $v\in K_\infty (f)$ and elements $\zeta ,\theta \in \mathbb{S}$ with $v_{1_{\h} }=1$, $\tilde{\lambda }^xv=\zeta ^{-1}v$, and $\tilde{\lambda }^zv=\theta ^{-1}v$. Equation \eqref{eq:actionH} shows that there exists an element $c\in\ell ^\infty (\mathbb{Z},\mathbb{R})$ such that $c_0=1$ and
	\begin{displaymath}
v_{[k_1,k_2,k_3]} = c_{k_2}\zeta ^{k_1}\theta ^{k_3}
	\end{displaymath}
for every $[k_1,k_2,k_3]\in \h$. Furthermore, if $h\in \mathbb{Z}[x^{\pm1},z^{\pm1}]\subset \mathbb{Z}[\h]$, then
	\begin{displaymath}
(\tilde{\rho }^hv)_{[k_1,k_2,k_3]}=h(\zeta \theta ^{-k_2},\theta )v_{[k_1,k_2,k_3]} = c_{k_2}h(\zeta \theta ^{-k_2},\theta )\zeta ^{k_1}\theta ^{k_3}
	\end{displaymath}
for all $[k_1,k_2,k_3]\in \h$. Since
	\begin{displaymath}
(\tilde{\rho }^{y^k}v)_{[k_1,k_2,k_3]}=v_{[k_1,k_2+k,k_3]}
	\end{displaymath}
for every $k\in \mathbb{Z}$, our hypothesis that $v\in K_\infty (f)$ yields that
	\begin{displaymath}
c_{k}g_0(\zeta \theta ^{-k},\theta )+c_{k+1}g_1(\zeta \theta ^{-k-1},\theta )=0\enspace \enspace \textup{for every}\enspace k\in \mathbb{Z}.
	\end{displaymath}
Hence
	\begin{equation}
	\label{eq:v}
v_{[k_1,k_2,k_3]} = (-1)^{k_2}e^{c_\theta (k_2,\zeta )}\zeta ^{k_1}\theta ^{k_3}
	\end{equation}
for every $(k_1,k_2,k_3)\in \mathbb{Z}^3$. Since $v\in \ell ^\infty (\h,\mathbb{R})$, there exists a constant $C\ge 0$ such that
	\begin{equation}
	\label{eq:bound}
c_\theta (n,\zeta )\le C
	\end{equation}
for every $n\in \mathbb{Z}$.

If $\theta ^p=1$ for some $p\ge1$, and if $a=c_\theta (p,\zeta )=\sum_{j=0}^{p-1}\psi _\theta (\zeta \theta ^{-j})$, then $c_\theta (kp,\zeta )=ka$ for every $k\in \mathbb{Z}$. As $c_\theta (n,\zeta )\le C$ for every $n\in \mathbb{Z}$ we obtain that $a=0$. Since
	\begin{displaymath}
0=\sum_{j=0}^{p-1}\psi _\theta (\zeta \theta ^{-j}) = \sum_{j=0}^{p-1}\log \bigl|\tfrac {g_0(\zeta \theta ^{-j}, \theta )}{ g_1(\zeta \theta ^{-j-1}, \theta )}\bigr| = \sum_{j=0}^{p-1}\log \bigl|\tfrac {g_0(\zeta \theta ^{-j}, \theta )}{ g_1(\zeta \theta ^{-j}, \theta )}\bigr| = \sum_{j=0}^{p-1}\phi _\theta (\zeta \theta ^{-j}) = \sum_{j=0}^{p-1}\phi _\theta (\zeta \theta ^j),
	\end{displaymath}
this proves (2).

Now suppose that $\theta $ in \eqref{eq:v} -- \eqref{eq:bound} is not a root of unity. We write $\xi \mapsto R_\theta \xi =\xi \theta$ for the rotation by $\theta $ on $\mathbb{S}$. Since $R_\theta $ is uniquely ergodic, the ergodic averages $\frac 1n \sum_{j=0}^{n-1}\psi _\theta \cdot \circ R_\theta ^{-j}$ converge uniformly to $\int \psi _\theta \,d\lambda $. If $\int \psi _\theta \,d\lambda \ne 0$ it follows that either $\lim_{n\to\infty }c_\theta (n,\xi )=\infty $ or $\lim_{n\to-\infty }c_\theta (n,\xi )=\infty $ \textit{uniformly in $\xi $}, in violation of \eqref{eq:bound}. This implies that $\int \phi _\theta \,d\lambda = \int \psi _\theta \,d\lambda =0$ and proves (1) in this special case.

Conversely, assume that $\theta $ is not a root of unity and that $\int \phi _\theta \,d\lambda = \int \psi _\theta \,d\lambda  =0$. If the cocycle $c_\theta $ is a coboundary, i.e., if there exists a continuous function $b\colon \mathbb{S}\longrightarrow \mathbb{R}$ such that $\psi _\theta (\xi )=b(\xi \theta ^{-1}) - b(\xi )$ for every $\xi \in \mathbb{S}$, then there obviously exists a constant $C\ge 0$ with $|c_\theta (n,\zeta )|<C$ for every $n\in \mathbb{Z}$ and every $\zeta \in \mathbb{S}$. The element $v$ defined by \eqref{eq:v} is thus bounded for every $\zeta \in \mathbb{S}$. Since it also lies in $K_\infty (f)$, $\alpha _f$ is nonexpansive.

If $c_\theta $ is not a coboundary, we set $Y=\mathbb{S}\times \mathbb{R}$ and define a skew-product transformation $S\colon Y\longrightarrow Y$ by setting $S(\xi ,t)=(\xi \theta ^{-1},t+\psi _\theta (\xi ))$ and hence $S^m(\xi ,t)=(\xi \theta ^{-m},t+c_\theta (m,\xi ))$ for every $m\in \mathbb{Z}$ and $(\xi ,t)\in Y$. Since $c_\theta $ is not a coboundary, but $\int \psi _\theta \,d\lambda =0$, the homeomorphism $S$ is topologically transitive by \cite[Theorems 14.11 and 14.13]{GH} or \cite[Theorems 1 and 2]{Atkinson}. By \cite[p. 38 f.]{Besicovitch} there exists a point $\xi \in \mathbb{S}$ such that the entire $S$-orbit of $(\xi ,0)$ is bounded above in the sense that there exists a constant $C\ge 0$ with $c_\theta (n,\zeta )\le C$ for every $n\in \mathbb{Z}$. Again we conclude that the point $v$ in \eqref{eq:v} lies in $K_\infty (f)$ and that $\alpha _f$ is therefore not expansive.

Finally, if $\int \phi _\theta \,d\lambda =\int \psi _\theta \,d\lambda =0$ for some $p$-th root of unity $\theta $, then the mean value theorem allows us to find a $\zeta \in \mathbb{S}$ with $c_\theta (p,\zeta )=0$. Then $(\zeta ,\theta )$ satisfies (2), so that $\alpha _f$ is nonexpansive. This completes the proof of the theorem.
	\end{proof}

	\begin{proof}[Proof of Corollary \ref{c:LS3}]
Consider the group automorphism $\tau \colon \h\longrightarrow \h$ satisfying $\tau (x)=y$, $\tau (y)=x$, and $\tau (z)=z^{-1}$. An element $f\in\mathbb{Z}[\h]$ satisfies the conditions \eqref{eq:form1} -- \eqref{eq:form2} in Theorem \ref{t:LS2} if and only if $\tilde{f}\coloneqq f\circ \tau $ satisfies the hypotheses of Corollary \ref{c:LS3}. For every $x\in X_f$ we define $\tilde{x}\in \mathbb{T}^{\h}$ by $\tilde{x}_\gamma =x_{\tau (\gamma )},\,\gamma \in \h$. The map $x\mapsto \tilde{x}$ sends $X_f$ to $X_{\tilde{f}}$ and intertwines the algebraic $\h$-actions $\alpha _f$ and $\tilde{\alpha }_{\tilde{f}}$, defined by $\tilde{\alpha }_{\tilde{f}}^\gamma =\alpha _{\tilde{f}}^{\tau (\gamma )}$. Then $\alpha _f$ is expansive if and only if the same is true for $\alpha _{\tilde{f}}$ (or, equivalently, for $\tilde{\alpha }_{\tilde{f}}$), and $f$ satisfies the condition in Remark \ref{r:LS2} if and only if $\tilde{f}$ does.
	\end{proof}

The method of Theorem \ref{t:LS2} could --- in principle --- be applied to an arbitrary nonzero $f\in \mathbb Z[\h]$. If
	$$
f=g_N(x,z)y^N+g_{N-1}(x,z)y^{N-1}+\ldots + g_1(x,z)y+g_0(x,z)
	$$
with $\mathsf{U}(g_0)=\mathsf{U}(g_N)=\varnothing $, then one can consider one-dimensional subspaces $V\subset K_\infty (f)$ which are invariant under the restriction of $\tilde{\lambda }$ to the subgroup $\Delta \subset \Gamma $ generated by $x$ and $z$, and express the growth rate of a nonzero element $v\in V$ in terms of a cocycle $M_\theta (n,\xi )$, which is a product of $N\times N$ matrices with polynomial entries in $\xi $ and $\theta $. However, matrix cocycles are considerably more difficult to analyse than scalar cocycles, so that this idea has not yet led to any significant progress for higher degree polynomials in $y$.

\section{Allan's local principle}\label{s:allan}

In this section we introduce Allan's local principle and use it to find an algebraic characterisation of invertibility of elements $\Z[\h]$ in the group algebra $\ell^1(\h,\C)$.

\subsection{General theory} Suppose that $\mathcal{A}$ is a unital Banach algebra with nontrivial center
	$$
C(\mathcal{A}) \coloneqq \bigl\{c \in \mathcal A \,:\, cb=bc \,\,\,\textup{for all }b \in \mathcal A\bigr\}.
	$$
For every Banach subalgebra $\mathcal{C}\subset C(\mathcal{A})$ we denote by $M_\mathcal{C}$ the space of maximal ideals of $\mathcal{C}$, equipped with its usual (compact) topology. For every ideal $m \in M_\mathcal{C}$ we denote by $\mathcal{J}_m$ the smallest closed two-sided ideal of $\mathcal{A}$ containing $m$ and consider the quotient map $\Phi_m\colon \mathcal{A}\longrightarrow \mathcal{A}/\mathcal{J}_m$ defined by $\Phi_m(a)= [a]_m \coloneqq a + \mathcal{J}_m$ for $a \in \mathcal{A}$. The quotient algebra $\mathcal{A}/\mathcal{J}_m$ is a unital Banach algebra with the norm
	\begin{displaymath}
\| [a]_m\|=\inf_{b\in \mathcal{J}_m}\|a+b\|_\mathcal{A},\enspace a\in \mathcal{A},
	\end{displaymath}
where $\|\cdot \|_\mathcal{A}$ is the norm on $\mathcal{A}$.

	\begin{theo}[Allan's local principle \protect{\cite{Allan}}]
	\label{t:allan}
Let $\mathcal{A}$ be a unital Banach algebra, and let $\mathcal{C}\subset \mathcal{A}$ be a closed central subalgebra which contains the identity element $1=1_\mathcal{A}$ of $\mathcal{A}$. An element $a \in \mathcal{A}$ is invertible in $\mathcal{A}$ if and only if $\Phi_m(a)$ is invertible in $\mathcal{A}/\mathcal{J}_m$ for every $m\in M_{\mathcal{C}}$.
	\end{theo}
We refer to \protect{\cite[Proposition 2.2.1]{RSS}} for a modern treatment of localisation methods and for a complete proof of Theorem \ref{t:allan}.
	\begin{proof}[Sketch of proof]
If $a\in\mathcal{A}$ is left (or right) invertible in $\mathcal{A}$, then $\Phi_m(a)$ is obviously left (resp. right) invertible in $\mathcal{A}/\mathcal{J}_m$ for every $m\in M_{\mathcal{C}}$.

For the opposite direction one can argue by contradiction. Assume that $\Phi_m(a)$ is left invertible in $\mathcal{A}/\mathcal{J}_m$ for every $m\in M_{\mathcal{C}}$, but that $a$ is  not invertible in $\mathcal{A}$. Let $\mathcal{M}$ be a maximal left ideal containing the proper left ideal $\{b a \,:\, b\in \mathcal A\}$, and let $m = \mathcal M \cap \mathcal C$, which is a maximal ideal of $\mathcal C$. Since $\sum_{k=1}^{n} a_k m_k b_k = \sum_{k=1}^{n} a_k b_k m_k \in \mathcal{M}$ for $ a_k, b_k \in \mathcal A$ and $m_k \in m\subset \mathcal C$, we conclude that $\mathcal J _m \subset \mathcal M$.

By assumption, $\Phi_m(a)$ is left invertible in $\mathcal{A}/\mathcal{J}_m$, i.e., there exists an element $b\in \mathcal A$ with $ba-1\in \mathcal J _m$. We now have a violation of the fact that $\mathcal M$ is a maximal left ideal: by definition of $\mathcal M$, $b a \in \mathcal M$; on the other hand, $ba-1 \in \mathcal M$ and hence $1 \in \mathcal M$.
	\end{proof}

\subsection{The commutative case}

In the commutative setting Allan's local principle reduces to a familiar criterion for invertibility. We reformulate the following well-known result from commutative Gelfand's theory in terms of Allan's local principle:

	\begin{coro}
	\label{c:commutative}
Let $\mathcal C$ be a unital commutative Banach algebra. An element $a\in \mathcal C$ is invertible in $\mathcal C$ if and only if $a \notin m$ for all $m \in M_\mathcal{C}$.
	\end{coro}

Since $\mathcal C$ is commutative we obtain the following simplifications: $\mathcal C=C(\mathcal C)$, $\mathcal J_m =m$ and $\mathcal C /m \simeq \C$ for every $m \in M_\mathcal{C}$ by the theorem of Gelfand-Mazur (\cite[Corollary 1.2.11]{RSS}). An element $a\in \mathcal C$ is invertible in $\mathcal C$ if and only if
	$$
\Phi_m(a)=a + m \ne m = \Phi_m (0)
	$$
for all maximal ideals $m \in M_\mathcal{C}$. Since the set $\{\Phi_m: m\in M_\mathcal{C}\}$ coincides with the set of multiplicative linear functionals from $\mathcal C$ to $\C$, Allan's local principle just says that $a \in \mathcal C$ is invertible if and only if there is no multiplicative functional vanishing at $a$. In particular, if $\mathcal C = \ell^1(\Z^d,\C)$, then $a \in \ell^1(\Z^d,\C)$ is invertible if and only if the Fourier-transform of $a$ does not vanish on $\mathbb{S}^d$.

\subsection{The group algebra $\ell^1(\h,\C)$}\label{ss:ZH} Next we use Allan's local principle for characterising invertibility for elements $f \in \ell^1(\h,\C)$. The center $C(\h)$ of the Heisenberg group is generated by $z$ and is isomorphic to the abelian group $(\Z,+)$. The space of maximal ideals of $\mathcal C (\ell^1(\h,\C))=\ell^1(C(\h),\C)$ is thus homeomorphic to $\widehat{\Z}\cong \mathbb{S}=\{\theta \in \mathbb{C}:|\theta |=1\}$ (cf. \cite{Folland}), where the maximal ideal corresponding to $\theta \in \mathbb{S}$ is given by
	$$
m_\theta \coloneqq \Bigl\{ f \in \ell^1 (\Z,\C): \hat{f}(\theta^{-1} ) = \sum\nolimits_{n \in \Z} f_n \theta ^{n}= 0\Bigr\}.
	$$

Before we state the main result of this section we have to determine the explicit form of the ideal $\mathcal J _\theta$ corresponding to $m_\theta$. Let us assume for the moment that $\mathcal A$ is a unital Banach algebra with a central $C^*$-subalgebra $\mathcal C$.
We know from \cite[Proposition 2.2.5]{RSS} that for any $m \in M_{\mathcal C}$, the two-sided closed ideal $\mathcal J _m$ containing $m$ is of the form
\begin{equation}\label{productofideals}
	\{ca \,:\, a\in\mathcal A \,\text{and}\, c\in m\}\,.
\end{equation}
Unfortunately, $\ell^1(\h,\C)$ does not possess a central $C^*$-subalgebra and we cannot expect $\mathcal J _\theta$ to be of the simple form of (\ref{productofideals}).

We write a typical element $f\in \ell^1(\h,\C)$ as
$$
	f=\sum_{(k,l,m)\in\Z^3} f_{(k,l,m)} x^ky^lz^m\,.
$$
Fix $\theta \in \s$. Let $\mathcal J _\theta$ be the subset of $\ell^1(\h,\C)$ which consists of all elements $f \in \ell^1(\h,\C)$ such that
	$$
		\sum_{(k,l,m)\in\Z^3} f_{(k,l,m)} x^ky^l\theta^m = 0_{\ell^1(\Z ^2,\C)}\,.
	$$
Equivalently, $\mathcal J _\theta$ is the set of $\ell^1$-summable formal series of the form
$$
	\sum_{(k,l)\in\Z^2} f_{(k,l)}(z) x^k y^l \,,
$$
where $f_{(k,l)}(z) \in m_\theta$.
The set $\mathcal J _\theta$ is a two-sided ideal since for all $f\in\ell^1(\h,\C)$ and $h\in\mathcal J _\theta$ the convolutions
\begin{equation*}
	 f \cdot h = \sum_{\gamma \in \h} f_\gamma \gamma \sum_{\gamma' \in \h} h_\gamma' \gamma'\quad\text{and}\quad
		h \cdot f = \sum_{\gamma \in \h} h_\gamma \gamma \sum_{\gamma' \in \h} f_\gamma' \gamma'
\end{equation*}
are just (infinite) linear combinations of elements in $\mathcal J_\theta$.
\begin{lemm}
	The set $\mathcal J _\theta$ is a closed subset of $\ell^1(\h,\C)$.
\end{lemm}
\begin{proof}
	Let $f \in \ell^1(\h,\C)$ be an element in the closure of $\mathcal J _\theta$ and $(f^{(n)})$ any sequence such that each term $f^{(n)}\in \mathcal J _\theta$
	and $\lim_{n \to \infty} f^{(n)}= f$  with respect to the $\ell^1$-norm.
	We will show that $f_{k,l}(z)$ lies in $m_\theta$ for all $(k,l)\in \Z^2$
	and hence that $\mathcal J _\theta$ is closed.

	The convergence with respect to the $\ell^1$-norm implies the convergence of
	$(f^{(n)}_{(k,l)}(z))$ to $f_{k,l}(z)$ in $\ell^1(\Z,\C)$, because
	$$
		\|f^{(n)}_{(k,l)}(z)-f_{(k,l)}(z)\|_{\ell^1(\Z,\C)} \leq \|f^{(n)}-f\|_{\ell^1(\h,\C)}\,.
	$$
	The terms $f^{(n)}_{(k,l)}(z)$ are elements in $m_\theta$, for all $n\in\N$, by the definition of $\mathcal J_\theta$.
	Since $m_\theta$ is a closed ideal in
	$\ell^1(\Z,\C)$, we can conclude that $f_{k,l}(z)\in m_\theta$.
\end{proof}

Moreover, every element in the closed two-sided ideal $\mathcal J_\theta$ can be approximated by finite linear combinations of elements in $m_\theta \cdot \ell^1(\h,\C)$. Therefore, $\mathcal J _\theta$ is the smallest closed two-sided ideal which contains $m_\theta$.

We note the following corollary of Allan's local principle.
	\begin{coro}
	\label{c:allan}
An element $a \in \mathcal{A}=\ell^1(\h,\C)$ is invertible if and only if there exists, for every $\theta \in \mathbb{S}=M_{\mathcal C (\mathcal{A})}$, an element $b_\theta \in \mathcal{A}$, such that
	$$
(a +\mathcal{J}_\theta ) \cdot (b_\theta  +\mathcal{J}_\theta )= (1 +\mathcal{J}_\theta )
	$$
or, equivalently, if $(ab_\theta -1) \in \mathcal{J}_\theta $ or $1\in a \cdot \mathcal{A} + \mathcal{J}_\theta $.
	\end{coro}
We use the following notation for the quotient norm on $\mathcal{A}/ \mathcal{J}_\theta $:
	\begin{equation}
	\label{eq:defnorm}
\|a\|_\theta \coloneqq \inf_{b\in \mathcal J_\theta } \|a + b\|_1.
	\end{equation}
For $\theta =1$, $\|a\|_\theta $ is not necessarily equal to the $\ell^1$-norm of $a$ for $a \in \ell^1(\h,\C)$. Hence we write $\|\cdot \|_{\theta =1}$ for the quotient norm on $\mathcal{A}/ \mathcal{J}_1$ to avoid confusion.

\subsection{Examples} In this subsection we apply Allan's local principle to certain elements $f\in \mathbb{Z}[\h]$ and compare its effectiveness with that of other conditions for (non)invertibility of $f$ in $\ell ^1(\h,\mathbb{C})$ (and hence for (non)expansiveness of $\alpha _f$) in the Theorems \ref{t:ER} or \ref{t:LS2}.

Before we start our discussion of concrete examples we recall the following elementary lemma:
	\begin{lemm}
	\label{l:neumann}
Let $\mathcal A$ be a unital Banach algebra, and let $a \in \mathcal A$ with $\| a\| <1$. Then $1-a$ is invertible in $\mathcal A$.
	\end{lemm}

	\begin{exam}
	\label{e:ex1}
Let $f=3+x+y+z$.

	\begin{enumerate}
	\item
This example was studied in \cite[Example 8.4]{ER}, where it was shown that $\alpha _f$ is expansive. In Example \ref{e:ER1} we saw that expansiveness of $\alpha _f$ is equivalent to showing that $\tilde{\pi }(f)$ has trivial kernel for every irreducible unitary representation $\pi $ of $\h $, and that the latter condition only has to verified for all one-dimensional representations of $\h $. Since $\tilde{\chi }(f) = \sum_{\gamma \in \h }f_\gamma \chi (\gamma )\ne0$ for every homomorphism $\chi \colon \h \longrightarrow \mathbb{S}$ we obtain that $\alpha _f$ is indeed expansive.
	\item
In Subsection \ref{ss:3+x+y+z} in the Appendix we will check explicitly that 
$$w =\sum_{M=0}^\infty (-1)^M(x\linebreak[0]+y)^M ((3+z)^{-1})^{M+1}$$ 
lies in $\ell ^1(\h ,\mathbb{R})$ and is the inverse of $f$. Compared to the other approaches this method is rather complicated and needs nontrivial combinatorial results.
	\item
Let us now establish expansiveness of $\alpha_f$ by applying Allan's local principle, using the notation of Corollary \ref{c:allan} with $\mathcal{A}=\ell ^1(\h,\mathbb{C})$. Consider $\theta \in \s \smallsetminus \{-1\}$, and let $\Phi_\theta \colon \mathcal A\to \mathcal A / \mathcal J _\theta $ be the corresponding factor mapping. Clearly,
	$$
\Phi_\theta (3+x+y+z) = \Phi_\theta (3+\theta +x+y) .
	$$
The invertibility of $\Phi _\theta (3+\theta +x+y)$ in $\mathcal A / \mathcal J _\theta $ follows immediately from Lemma \ref{l:neumann}, since
	$$
\|3+\theta \|_\theta = |3+\theta | > 2 \ge  \| x+y \|_\theta .
	$$
For $\theta =-1$, Lemma \ref{l:neumann} does not apply directly, and this case has to be treated separately. Let us show that $\Phi_{-1}(3+x+y+z)$ is invertible as well. Indeed, $3+x+y$ is invertible in $\mathcal{A}$ by Lemma \ref{l:neumann}, and
	\begin{align*}
\Phi_{-1}(3+x+y+z) &= \Phi_{-1}(3+x+y-1)
	\\
&= \Phi_{-1}((3+x+y)(1-(3+x+y)^{-1})).
	\end{align*}
In the notation of \eqref{eq:q-binomial}, $\Phi_{-1}(3+x+y)$ has an inverse with norm
	\begin{align}
\| (3+x+y)^{-1}\|_{-1} &= \biggl\| \frac {1}{3}\sum_{n=0}^{\infty} \frac{1}{3^{n}} (-1)^n(x+y)^n\biggr\|_{-1} \nonumber
	\\ \label{e:comparenorms}
	&  \leq \frac {1}{3}\sum_{n=0}^{\infty} \frac{1}{3^{n}}\bigl\|(x+y)^n\bigr\|_{-1}	< \frac {1}{3}\sum_{n=0}^{\infty} \frac{1}{3^{n}}\bigl\|(x+y)^n\bigr\|_{1}
	\\
&\leq  \frac 1 3 \sum_{n=0}^\infty \left ( \frac 2 3 \right ) ^n =1	\nonumber
	\end{align}
We can thus apply Lemma \ref{l:neumann} in the quotient Banach algebra $\mathcal A / \mathcal J _{-1}$, furnished with the norm (\ref{eq:defnorm}).
One can easily see that the inequalities hold by comparing the term
$$
\|(x+y)^2\|_{-1}=\|(x^2+xy+yx+y^2)\|_{-1}=\|(x^2+y^2)\|_{-1}=2$$ with the term $\|(x+y)^2\|_{1}=4$ which appear in the inequality (\ref{e:comparenorms}).
	\item
The cocycle approach of Theorem \ref{t:LS2} again provides the easiest method to establish expansiveness: For every $(\xi ,\theta )\in \mathbb{S}^2$ we put $g_\theta (\xi )=\log(3+\xi +\theta )$. Then $\int g_\theta (\xi )\,d\nu (\xi ) \linebreak[0]\ne0$ for every probability measure $\nu $ on $\mathbb{S}$ which is invariant and ergodic under the rotation $R_\theta \colon \mathbb{S}\longrightarrow \mathbb{S}$. By Remark \ref{r:LS2} this proves that $\alpha _f$ is expansive.
	\end{enumerate}
	\end{exam}

	\begin{exam}
	\label{e:ex2}
Let $f= 3+x+y-z=3\cdot 1_{\h} +x+y-z$.
	\begin{enumerate}
	\item
In \cite{ER} it was shown that $\alpha _f$ is nonexpansive (see Example \ref{e:ER1}).
	\item
The formal inverse
	$$
w =\sum_{M=0}^\infty (-1)^M(x+y)^M ((3-z)^{-1})^{M+1}
	$$
of $f$ lies in $\ell^\infty(\h,\C)$. Since there is no alternating sign in the expansion of $(3-z)^{-1}$, $w$ cannot be summable, and so we conclude that $f$ is not invertible by Corollary \ref{c:expa} (2).
	\item
For the application of Allan's local principle the discussion is almost identical to that of Example \ref{e:ex1} (3), except that the roles of $+1$ and $-1$ are interchanged. For $\theta =1$ the analysis reduces to the abelian case, since $\Phi _1(\mathcal{A})$ is abelian. Here the Fourier transformation of $2+x+y$ vanishes on $\s^2$, i.e., $2+e^{2 \pi i s}+e^{2 \pi i t}=0$ for $s=t=1/2$. Therefore, $\Phi _{1}(3+x+y-z)$ is not invertible.

	\item
The cocycle approach: put $g_\theta (\xi )= \log (3+\xi-\theta )$, set $\theta =1$, and denote by $\nu $ the probability measure on $\mathbb{S}$ concentrated on $\xi =-1$. Then $\int g_1 d\nu =0$, and $\alpha _f$ is nonexpansive by Remark \ref{r:LS2}.
	\end{enumerate}
	\end{exam}

	\begin{exam}
	\label{e:ex4}
Consider polynomials of the form $f=|a|+|b|+|c| + a\cdot x+b\cdot y+c\cdot z\in\mathbb{Z}[\h]$ (cf. Example \ref{e:ER1}).
	\begin{enumerate}
	\item
If $c=0$ and $|a|+|b|\ne0$, then $\alpha _f$ is nonexpansive either by Theorem \ref{t:LS2} or by Corollary \ref{c:LS2}. If $c\ne0$, but $a=b=0$, then $\alpha _f$ is clearly nonexpansive. 
If $c\ne0$, $ab=0$, but $|a|+|b|>0$, then $\alpha _f$ is again nonexpansive either by Theorem \ref{t:LS2} or by Corollary \ref{c:LS2}.
	\item
If $c\ne0$ and $ab\ne0$, we can use Theorem \ref{t:LS2} to show that $f$ is invertible if and only if $c>0$ (cf. \cite[Example 8.4]{ER} and Example \ref{e:ER1}).
	\item
Now suppose that $c>0$, $ab\ne0$, and $|a|+|b|>2$. We apply Allan's local principle to show expansiveness.

For $\theta \in \s \smallsetminus \{-1\}$, Lemma \ref{l:neumann} shows that $\Phi_\theta (f)$ is invertible. For $\theta =-1$ we set $A\coloneqq |a|+|b|+\text{sign}(a)\cdot x+\text{sign}(b)\cdot y$ and $B\coloneqq (a-\text{sign}(a)) \cdot x + (b-\text{sign}(b)) \cdot y$, and write $\Phi_{-1}(f)=\Phi _{-1}(A+B)$ as
	\begin{align}
	\label{trick}
&\Phi_{-1}(|a|+|b|+a \cdot x +b \cdot y ) = \Phi_{-1}\left ( A(1+A^{-1}B) \right ).
	\end{align}
Clearly, $ \Phi_{-1}(A) $ is invertible and we have the norm-estimate
	\begin{equation}
	\label{trick2}
	\begin{aligned}
\| A^{-1}\|_{-1} & \leq \biggl\| \frac{1}{|a|+|b|}\sum_{n=0}^\infty \frac{1}{(|a|+|b|)^n}\sum_{k=0}^n \qbinomchi{n}{k} x^k y^{n-k} \biggr\|_{-1}
	\\
& < \frac{1}{|a|+|b|}\sum_{n=0}^\infty \frac{2^n}{(|a|+|b|)^n} = \frac{1}{|a|+|b|-2}.
	\end{aligned}
	\end{equation}
Furthermore,
	$$
\|B\|_{-1} \leq |a|+|b|-2
	$$
and one can apply Lemma \ref{l:neumann}, which completes the proof that $f$ is invertible in $\mathcal{A}=\ell ^1(\h,\mathbb{C})$.
	\end{enumerate}
	\end{exam}

	\begin{exam}
	\label{e:ex5}
The arguments used in Example \ref{e:ex4} (3) can be applied to more general $f\in \mathbb{Z}[\h]$ with $f_{1_{h}}=\sum_{1_{\h}\ne \gamma \in \h}|f_\gamma|$. 		\begin{enumerate}
	\item
Set $f=4+x+y+x^{-1}+z$. Again we can use Lemma \ref{l:neumann} for all $\theta \not = -1$. For $\theta =-1$ we use the same method as in (\ref{trick}) and (\ref{trick2}). First we write $3+x+y+x^{-1}$ as $(3+x+y)(1+(3+x+y)^{-1}x^{-1} )$. Since $3+x+y$ is invertible and the quotient-norm of the inverse is smaller than $1$ we can use Lemma \ref{l:neumann} to conclude that $3+x+y+x^{-1}$ is invertible. The same argument can be used to show that $5+x+y+x^{-1}+y^{-1}+z$ is invertible.
	\item
Let $f=3+x^2+y+z^2$, which is just a slight modification of the previous example. Again we can verify that $\Phi_\theta (f)$ is invertible for $\theta \in\s\smallsetminus\{-1,e^{\pi i/2}, e^{3 \pi i/2}\}$.

For $\theta =-1$,
	\begin{align*}
\Phi_{-1}(3+x^2+y+z^2)&=\Phi_{-1}(3+x^2+y+z^2 -(z+1)(z-1))
	\\
&=\Phi_{-1}(3+x^2+y+1)=\Phi_{-1}(4+x^2+y),
	\end{align*}
which is obviously invertible. Let $\chi \in \{e^{\pi i/2}, e^{3 \pi i/2}\} $, then $x^2$ and $y$ are anti-commuting and therefore we can use the same method as in example \ref{e:ex1} or (\ref{trick}) and (\ref{trick2}). We can now conclude that $3+x^2+y+z^2$ is invertible.
	\item
Let $f=3+xy+yx+z$. Since the invertibility of $\Phi_\theta (3+xy+yx+z)$ for $\theta \ne -1$ follows from Lemma \ref{l:neumann}, we only have to verify the invertibility of $\Phi_{-1}(3+xy+yx+z)$. Indeed,
	$$
\Phi_{-1}(3+xy+yx+z)=\Phi_{-1}(2+xy-xy)=\Phi_{-1}(2).
	$$
Therefore, the algebraic dynamical system $(X_f,\alpha_f)$ which is defined by $f$ is expansive.

Example \ref{e:ER1} provides another way of showing that $\alpha _f$ is expansive: for every group homomorphism $\chi \colon \h\longrightarrow \mathbb{S}$ we have that $\chi (z)=1$ and hence $\tilde{\chi }(f)=4-2\xi (x)\xi (y)\ne0$.
	\end{enumerate}
	\end{exam}

We continue with an application of Corollary \ref{c:LS4}.

	\begin{exam}
	\label{e:mistake}
Set $f=2+x+y+z$. In \cite[Example 10.6]{ER} it is stated (essentially without proof) that $\alpha_f$ is expansive. Here we show that this is not the case.

Put $\theta =-1$, and consider the $R_\theta $-invariant probability measure $\nu _\zeta $ on $\mathbb{S}$ which is equidistributed on the set $S_\zeta =\{\zeta ,-\zeta \}\subset \mathbb{S}$ with $\zeta \in \mathbb{S}\smallsetminus \{\pm 1\}$. Following the notation of \eqref{eq:form1} we write $f$ as $g_1y+g_0$ with $g_1(\xi ,\theta )=1$ and $g_0(\xi ,\theta )=2+\xi +\theta $. Then $g_0(\xi ,-1)\ne 0$ for $\xi \in S_\zeta $, and
	\begin{align*}
\int  |\phi _\theta | \,d\nu _\zeta &= \tfrac12\cdot (\log |g_0(\zeta ,-1)| + \log |g_0(-\zeta ,-1)|)
	\\
&= \tfrac12\cdot (\log |1+\zeta | + \log |1-\zeta |) = \tfrac12 \log |1-\zeta ^2|.
	\end{align*}
By setting $\zeta =e^{\frac{2\pi i}{12}}$ we obtain that $|1-\zeta ^2|=1$ and hence $\int  |\phi _\theta | \,d\nu _\zeta =\tfrac12 \log |1-\zeta ^2|=0$. According to Corollary \ref{c:LS4} this shows that $\alpha _f$ is nonexpansive.

	\end{exam}

	\begin{exam}
	\label{e:last}
For $f=3+x^2+y^2-z^4$ we shall find a $4$-dimensional unitary representation $\pi$ of $\h$ for which $\tilde{\pi }(f)$ has a nontrivial kernel. By Theorem \ref{t:ER}, this implies that $\alpha _f$ is nonexpansive.

Let $\theta = e^{2\pi i/4}$. The elements $x^2$ and $y^2$ commute in $\ell^1(\h,\C)/\mathcal{J}_\theta $: indeed, $x^2y^2 = y^2x^2z^4$, and $\Phi_\theta (z^4)=\Phi_\theta (1) =1$.

Put
	\begin{align}
\pi (x)=U=
\left(\begin{smallmatrix}
1\hphantom{^3}& 0\hphantom{^3} & 0\hphantom{^2} & 0\hphantom{^3} \\
0\hphantom{^3}& \theta\hphantom{^3}  &0\hphantom{^2}&0\hphantom{^3} \\
0\hphantom{^3}&0\hphantom{^3}&\theta ^2&0\hphantom{^3}\\
0\hphantom{^3} &0\hphantom{^3}& 0\hphantom{^2}&\theta ^3\\
\end{smallmatrix}\right),\qquad
\pi (y)=V=\left(\begin{smallmatrix}
0&1&0&0\\0&0&1&0\\0&0&0&1\\1&0&0&0
\end{smallmatrix}\right),\qquad
 \pi (z)=\theta I_4,
\end{align}
where $\theta =e^{2\pi i/4}$, and where $I_4$ is the $4\times 4$ identity matrix. This assignment defines an irreducible $4$-dimensional unitary representation $\pi$ of $\h$ which is trivial on the subgroup $Z(4)\coloneqq \{z^{4n}:n\in \mathbb{Z}\}\subset \h$.

The representation $\tilde{\pi }$ of $\ell ^1(\h,\mathbb{C})$ induced by $\pi $ has the property that $\det \tilde{\pi }(f) = \det ( 2 I_4 + U^2 +V^2)=0$.

In order to emphasise the connection with Allan's local principle we note that the triviality of $\pi $ on $Z(4)$ implies that $\tilde{\pi }$ vanishes on the ideal $\mathcal{J}_\theta $ and may thus be viewed as a representation of $\ell^1(\h,\C)/\mathcal{J}_\theta $. In this representation, $\tilde{\pi }(\Phi _\theta (f))=\tilde{\pi }(f)$ has a nontrivial kernel, which implies that $\Phi _\theta (f)$ is noninvertible in $\ell^1(\h,\C)/\mathcal{J}_\theta $. Hence $f$ is noninvertible in $\ell^1(\h,\C)$.
	\end{exam}

We end this section with a last application of Allan's local principle.
Let us write a typical element in $\Z[\h]$ as $f=\sum_{(m,n)\in\Z^2} f_{(m,n)}(z)x^m y^n$, where $f_{(m,n)}(z)$ are Laurent polynomials in $z$. Then Allan's
local principle in combination with Lemma \ref{l:neumann} implies the following lemma.

\begin{lemm}\label{l:AN}
	Assume that $f\in\Z[\h]$ and that for all $\theta\in \s$ there exists an element $(k,l)\in \Z^2$ such that
	$$|f_{(k,l)}(\theta)|> \sum_{(m,n)\in\Z^2\smallsetminus\{(k,l)\}} |f_{(m,n)}(\theta)|\,,$$ then $f$ is invertible in $\ell^1(\h,\C)$.
\end{lemm}

With the previous Lemma we can easily find examples which are not covered by Example \ref{e:ER1} nor the cocycle method which is treated in Theorem \ref{t:LS2}.
\begin{exam}
Let $f= g(z)+\gamma_1+\gamma_2+\gamma_3+\gamma_4$ with $\gamma_1,\gamma_2,\gamma_3,\gamma_4\in\h \smallsetminus \{z^k:k\in \Z \}$ such that $f$ is not linear in $x$ or $y$, where
$g(z)=5+z-z^{-1}+z^2-z^{-2}$.
Clearly, $f$ has no dominating coefficient and neither Example \ref{e:ER1} nor Theorem \ref{t:LS2} can be applied for this example.
For all $\theta \in \s$ the absolute value $|g(\theta)|>4= \|\gamma_1+\gamma_2+\gamma_3+\gamma_4\|_1$.
Hence, $f^\theta= g(\theta)+\gamma_1+\gamma_2+\gamma_3+\gamma_4$ has a dominating coefficient for all
$\theta \in \s$. We apply Lemma \ref{l:AN} to conclude that
$f^\theta$ is invertible in  $\ell^1(\h,\C) / \mathcal J _\theta$, for all
$\theta \in \s$.
\end{exam}

\begin{rema}
In all the examples presented here we are considering nonexpansive principal action $\alpha _f,\,f\in\mathbb{Z}[\h]$, for which $\tilde{\pi }(f)$ has nontrivial kernel for some finite-dimensional representation $\pi $ of $\h$. In general one may also have to consider infinite-dimensional irreducible unitary representations of $\h$ when applying Theorem \ref{t:ER}.
\end{rema}

	\section{Homoclinic points}
	\label{s:section4}

In this section we construct a summable homoclinic point of the nonexpansive $\h$-action $(X_f,\alpha_f)$ with $f= 2-x^{-1}-y^{-1} \in \Z[\h]$ by using a multiplier method first introduced in \cite{sv} and further studied in \cite{lsv,lsv2} for the $\Z^d$-case.

Let $h \in \Z[\h]$ be noninvertible in $\ell^1(\h,\mathbb{R})$, but assume that there exists an element $h^\sharp\in \R^{\h} \smallsetminus \ell^1(\h,\mathbb{R})$ with $h^\sharp\cdot h= h\cdot h^\sharp=1\in \mathbb{Z}[\h]$. By abusing terminology we say that $h$ has a formal  inverse $h^\sharp$ in $\R^{\h}$. 
%
%
A \textit{central multiplier} of $h^\sharp$ is a central element $g\in \Z[\h]$ such that $g \cdot h^\sharp=h^\sharp\cdot g\in \ell^1(\Gamma,\mathbb{R})$. The reason why we focus on \textit{central} multipliers in $\Z [\Gamma]$ in the noncommutative setting will become clear in the next section, where we construct coding maps.

If we are able to find a central multiplier $g$ of the formal inverse $f^\sharp$ of $f$, then the point $x^{(g)}=(x^{(g)}_\gamma )_{\gamma \in \h }\in \mathbb{T}^{\h}$, defined by
	\begin{equation}
	\label{eq:fundamental}
x_\gamma ^{(g)}= ((f^\sharp)^*\cdot g^*)_\gamma \;(\textup{mod}\, 1)\enspace \textup{for every}\enspace \gamma \in \h,
	\end{equation}
is an absolutely summable point in $X_f$. Hence this construction leads to a summable homoclinic point of $X_f$.

Let us recall some basic notions about homoclinic points.

	\begin{defi}
	\label{d:homoclinic}
Let $\alpha $ be an algebraic action of a countable group $\Gamma $ on a compact abelian group $X$. A point $x \in X$ is \textit{homoclinic} if $\alpha ^\gamma x \rightarrow 0_X$ as $\gamma  \rightarrow \infty$. The set of homoclinic points forms a group, which we denote by $\Delta _\alpha (X)$.
	\end{defi}
In \cite[Theorem 5.6]{cl} it was shown that if $(X_f, \alpha _f)$ is expansive, then the group $\Delta _{\alpha _f}(X_f)$ is equal to the group of \textit{summable} homoclinic points, i.e.\
	$$
\Delta^1 _{\alpha _f}(X_f) \coloneqq \biggl\{ x \in \Delta _{\alpha _f}(X_f) \,:\, \sum_{\gamma \in \Gamma} \dT x_\gamma\dT<\infty\biggr\},
	$$
where $\dT t\dT$ denotes the shortest distance of a point $t\in \mathbb{T}$ from $0$.

The elements of $\Delta^1 _{\alpha _f}(X_f)$ play an essential role for constructing symbolic covers of principal algebraic dynamical systems. However, if $(X_f,\alpha_f)$ is nonexpansive, then it is in general not clear whether $\Delta_{\alpha _f}^1(X_f)$ is trivial or not (cf. \cite{LS1}, \cite{lsv}, \cite{lsv2}).

It is easy to write a formal inverse ${f^*}^\sharp$ of $f^*= 2-x-y$ in terms of $q$-binomial coefficients:
	\begin{equation}
	\label{eq:fundsol}
{f^*}^\sharp= \sum_{n=0}^{\infty} \frac {1}{2^{n+1}} (x+y)^n = \sum_{n=0}^{\infty} \frac {1}{2^{n+1}} \sum_{k=0}^n x^k y^{n-k}\qbinom{n}{k}
	\end{equation}
with $q=z^{-1}$. Then ${f^*}^\sharp\in \ell^\infty(\h,\mathbb{R})$ and ${f^*}^\sharp\cdot f^* = f^* \cdot {f^*}^\sharp=1$.

Since $q$-binomials are polynomials in $q$ with nonnegative coefficients, multiplying with $h=(1-z^{-1})$ will reduce the norm. The reduction is not sufficient for $\ell^1$-summability of $h\cdot {f^*}^\sharp$. Nevertheless one can explicitly construct an element in $\Delta^1 _{\alpha _f}(X_f)$.

	\begin{theo}
	\label{t:mainmulti}
Let ${f^*}^\sharp$ be the formal inverse of $f^*=2-x-y$ given by \eqref{eq:fundsol}. Then $(1-z^{-1})^2\cdot {f^*}^\sharp \in\ell^1(\h,\mathbb{R})$. Hence the point $x^{((1-z)^2)}$ in \eqref{eq:fundamental} is a summable homoclinic point in $X_f$.
	\end{theo}

In order to prove this theorem we will show that
	\begin{equation}
	\label{eq:estimate1}
\bigl \| {f^*}^\sharp\cdot (1-q)^2\bigr \| _1 \leq \sum_{n=0}^{\infty} \frac {1}{2^{n+1}} \biggl\| \sum_{k=0}^n x^k y^{n-k}\cdot (1-q)^2 \qbinom{n}{k} \biggr\|_1 < \infty
	\end{equation}
with $q= z^{-1}$. The key result we use for the proof of Theorem \ref{t:mainmulti} is the following estimate, which will be established in the appendix.
	\begin{theo}
	\label{MainEstimate}
Let $n >0$, then
	$$
\sum_{k=0}^n\,\biggl\|\qbinom{n}{k}\cdot(1-q)^2\biggr\|_1 = \mathcal{O}\Bigl( 2^n \frac{\log^2 (n)}{n^2}\Bigr).
	$$
	\end{theo}

	\section{Symbolic covers}
	\label{s:symbolic}
Suppose that $f \in \Z [\h ]$ is not invertible in $\ell^1(\h,\mathbb{R})$, but that $\Delta^1_{\alpha _f}(X_f) \ne \{0\}$. We construct shift-equivariant, surjective group homomorphisms $\xi \colon \ell^\infty (\h,\Z)\longmapsto X_f$ whose restriction to $\{-\|f\|_1+1,\ldots,\|f\|_1-1\}^{\h}$ is still surjective. By using such homomorphism we will find symbolic covers and establish a specification property for $X_f$.

\subsection{Coding maps and specification} Before defining the coding maps let us fix some notation. Let $\eta\colon \ell^\infty (\h,\R) \longrightarrow \T^{\h}$ be the group homomorphism defined by
	$$
\eta(w)_\gamma = w_\gamma \quad (\textup{mod}\, 1)
	$$
for every $w \in \ell^\infty (\h,\R)$ and $\gamma \in \h$. The map $\eta$ is obviously left and right shift-equivariant: for every $\gamma \in \h $,
	$$
\eta \circ \tilde{\lambda }^\gamma = \lambda^\gamma \circ \eta,\quad \eta \circ \tilde{\rho }^\gamma = \rho^\gamma \circ \eta.
	$$

Following \cite{LS1}, \cite{ES} and \cite{DS} we define the \textit{linearisation} of $X_f$ as the $\tilde{\lambda }$-invariant set
	$$
W_f \coloneqq \eta^{-1}(X_f) = \{w \in \ell^\infty (\h,\R)\, :\, w \cdot f^* \in \ell^\infty (\h,\Z)\}.
	$$

Let $f^\sharp \in \R^{\h} $ be a formal inverse of $f$ with $f \cdot f^\sharp =f^\sharp \cdot f=1$ and assume that $g$ is a central element in $\Z[\h]$ is such that $g\cdot f^\sharp = f^\sharp \cdot g\in \ell^1(\h,\R)$. The maps $\tilde{\xi_g}\colon \ell^\infty(\h,\Z) \longrightarrow \ell^\infty(\h,\R)$ and $\xi_g: \ell^\infty(\h,\Z) \longrightarrow \T^{\h}$, defined by
	\begin{equation}
	\label{eq:xig}
\tilde{\xi_g} (v) = v\cdot (g^*\cdot (f^\sharp)^*) = \rho ^{f^\sharp\cdot g}(v) = \rho ^{g\cdot f^\sharp}(v)\quad \textup{and}\quad \xi _g(v)=\eta (\tilde{\xi_g} (v))
	\end{equation}
are left shift-equivariant group homomorphisms:
	$$
\tilde{\xi_g} \circ \tilde{\lambda }^\gamma = \tilde{\lambda }^\gamma \circ \tilde{\xi_g} \quad \text{and} \quad \xi_g \circ \tilde{\lambda }^\gamma = \alpha_f ^\gamma \circ \xi_g ,
	$$
for every $\gamma \in \h$.

	\begin{theo}
	\label{t:codingmaptheo}
Let $f$ be an irreducible element in $\Z [\h]$ which has a formal inverse $f^\sharp\in \R^{\h}$. Furthermore, let $g \in \Z[\h]$ be a central element such that $g \not \in \Z [\h]f$ and $g\cdot
f^\sharp \in \ell^1(\h,\R)$. Then the following holds.
	\begin{enumerate}
	\item
	\label{weakcont}
$\xi_g$ is continuous in the weak$^*$-topology on closed, bounded subsets of $\ell^\infty(\h,\Z)$;
	\item
	\label{wwww}
$\xi_g (\{ 0,\ldots ,\|f\|_1-1 \} ^{\h} )=\xi_g (B^\infty _{\|f\|_1/2}) = \xi_g (\ell^\infty (\h,\Z))= X_f$, where $B^\infty_K$ denotes the closed ball of radius $K$ in $\ell^\infty(\h,\Z)$.
	\end{enumerate}
	\end{theo}

For the proof of Theorem \ref{t:codingmaptheo} we need a lemma.

	\begin{lemm}
	\label{l:Hayes}
Let $f\in \mathbb{Z}[\h]$ be irreducible and $g\in \mathbb{Z}[z^{\pm1}]$ a central element of $\mathbb{Z}[\h]$. If $h\in \mathbb{Z}[\h]$, and if neither $g$ nor $h$ are elements in $\mathbb{Z}[\h] f$, then $gh$ is not in $\mathbb{Z}[\h]f$.
	\end{lemm}

	\begin{proof}
We argue indirectly and assume that neither $g$ nor $h$ are divisible by $f$, but $gh=wf$ for some $w\in \mathbb{Z}[\h]$. Choose a decomposition $g=g_1\cdots g_k$ of $g$ into irreducible elements of $\mathbb{Z}[z^{\pm1}]$ such that no $g_i$ is a unit in $R_1$.

Our hypotheses guarantee that $f$ is not divisible by $g_k$ (i.e., that $f\notin \mathbb{Z}[\h]g_k$). By \cite[Proposition 3.3.1]{Hayes}, $w$ must be divisible by $g_k$, i.e., $w=w_{k-1}g_k$ for some $w_{k-1}\in \mathbb{Z}[\h]$. Since $\mathbb{Z}[\h]$ has no nontrivial zero-divisors \cite[Example 8.16.5(g)+Proposition 8.16.9]{CT}, we obtain the equation $g_1\cdots g_{k-1}h=w_{k-1}f$. By repeating this argument inductively we obtain an element $w_0\in \mathbb{Z}[\h]$ with $h=w_0f$, in violation of our assumption on $h$.
	\end{proof}

	\begin{rema}
	\label{r:Hayes}
We require Lemma \ref{l:Hayes} in the special case where $g=(z-1)^2$. The following simple-minded argument suffices for this purpose.

For every $\phi \in \mathbb{Z}[\h]$ set $\tilde{\phi }=\phi (x,y,1)\in \mathbb{Z}[x^{\pm1},y^{\pm1}]$, where $x$ and $y$ commute (in the notation of Subsection \ref{ss:ZH}, $\tilde{\phi }$ is the image of $\phi $ in $\Phi _1(\Z [\h])= \mathbb{Z}[\h]/\mathbb{Z}[\h](z-1)$). Then $\tilde{\phi }=0$ if and only if $\phi $ is divisible by $z-1$.

Assume that $f$ is not divisible by $(z-1)$. If there exist $h,w\in \mathbb{Z}[\h]$ such that $gh=wf$, then we claim that $h\in \mathbb{Z}[\h]f$. Indeed, since $gh=wf$, we have that $\tilde{g}\tilde{h}=0=\tilde{w}\tilde{f}$. As $\tilde{f}\ne0$ and $\mathbb{Z}[x^{\pm1},y^{\pm1}]$ has no nontrivial zero-divisors, $\tilde{w}$ must be equal to $0$. Put $w'=w/(z-1)$, $g'=(z-1)$, and apply the argument above to the equation $g'h=(z-1)h=w'f$. Then $w'$ is divisible by $z-1$. We set $w''=w/(z-1)^2$ and obtain that $h=w''f$, i.e., that $h\in \mathbb{Z}[\h]f$.
	\end{rema}

	\begin{proof}[Proof of Theorem \ref{t:codingmaptheo}]
The proof of (\ref{weakcont}) can be found in \cite[Proposition 4.2]{DS}. In order to prove (\ref{wwww}) we argue as in \cite[Lemma 4.5]{DS}: fix $x \in X_f$ and choose $w \in W_f$ such that $\eta(w) =x$ and $-\frac12 \leq w_\gamma < \frac12$ for every $\gamma \in \h$. If $v = w \cdot f^*$, then $v \in \ell^\infty (\h,\Z)$, $-\|f\|_1/2 \le v_\gamma  <\|f\|_1/2$ for every $\gamma \in \Gamma $, and $\xi_g (v)= \rho ^g x$. Therefore,
	$$
\rho ^g (X_f) \subset \xi_g (B^\infty _{\|f\|_1/2}) \subset \xi_g (\ell^\infty (\h,\Z))\subset X_f.
	$$
We will show that $\rho ^g(X_f) = X_f$. Clearly, $\rho ^gx=\lambda ^{g^*}x$ for every $x\in \mathbb{T}^{\h}$, since $g$ is central. If $\rho ^g(X_f)\subsetneq X_f$, then \eqref{eq:Xf} shows that there exists an element $h\in \mathbb{Z}[\h]\smallsetminus \mathbb{Z}[\h]f$ which annihilates $\rho ^g(X_f)=\lambda ^{g^*}(X_f)$, and which therefore satisfies that $hg\in \mathbb{Z}[\h]f$. Since $f$ is irreducible and not divisible by $(z-1)$, $g\in \Z[\h]f$ by either Lemma \ref{l:Hayes} or Remark \ref{r:Hayes}. This contradiction shows that $\rho ^g(X_f) = X_f$.

Finally we write $u=\textup{Int}(\|f\|_1/2)$ for the integral part of $\|f\|_1/2$ and denote by $\tilde{u}\in \ell ^\infty (\h,\mathbb{Z})$ the constant configuration with $\tilde{u}_\gamma =u$ for every $\gamma \in\h$. Then
	$$
\xi_g (\{ 0,\ldots ,\|f\|_1-1\} ^{\h} ) = \xi_g (B^\infty _{\|f\|_1/2} + \tilde{u}) =\xi_g (B^\infty _{\|f\|_1/2}) + \xi_g(\tilde{u}) = X_f,
	$$
which completes the proof of (\ref{wwww})).
	\end{proof}

	\begin{rema}
	\label{r:noncentral}
The reason for restricting attention to \textit{central} multipliers of $f^\sharp$ is the following: if a noncentral element $g \in \Z[\h]$ is such that
 $ f^\sharp\cdot g \in \ell ^1(\h,\mathbb{R})$, then the map $\xi_g$ given
 by \eqref{eq:xig} 
 \begin{equation}
	\label{eq:xignew}
 {\xi_g} (v) = v\cdot (f^\sharp\cdot g)^* \quad (\ \text{mod 1}) = \rho ^{f^\sharp\cdot g}(v) \quad (\ \text{mod 1})
	\end{equation}
	is still a well-defined coding map from $\ell^\infty(\h,\Z)$ into $X_f$.
 However, severe technical obstacles arise in Lemma \ref{l:Hayes}
 and elsewhere, e.g., for non-central $g$, $\rho^g(X_f)$ is not necessarily a subset of  $X_f$.
 	\end{rema}

We are now ready to state the following specification theorem.

	\begin{theo}
	\label{t:specification}
Let $f$ be an irreducible element in $\Z [\h]$ which has a formal inverse $f^\sharp \in \R^{\h}$. Suppose that there exists a central element $g \in \Z[\h]\smallsetminus \mathbb{Z}[\h]f$ with $f^\sharp \cdot g \in \ell^1(\h,\R)$. Then we can find, for every $\varepsilon > 0$, a nonempty finite subset $F_\varepsilon$ of $\Gamma$ with the following property: if $F_1,F_2$ are subsets of $\h$ with
	\begin{equation}
	\label{specification1}
F_\varepsilon F_1 \cap F_\varepsilon F_2=\varnothing ,
	\end{equation}
then there exists, for every pair of points $x^1,x^2$ in $X_f$, a point $y \in X_f$ with
	\begin{equation}
	\label{specification2}
\dT x^j_\gamma - y_\gamma \dT< \varepsilon \quad\text{for every} \quad \gamma \in F_j,\enspace \textup{and}\enspace  j=1,2.
	\end{equation}
	\end{theo}

	\begin{proof}
The proof follows from the proof of \cite[Theorem 4.4.]{DS} by replacing $w_{f^*}$ with ${f^*}^\sharp \cdot g^*$ and $x^j$ by $y^j = \rho_g x^j$, $j=1,2$.
	\end{proof}

\subsection{A symbolic cover of $X_f$ with $f=2-x^{-1}-y^{-1}$}

	\begin{defi}
	\label{d:cover}
Let $\Gamma $ be a countably infinite discrete amenable group, and let $f\in \mathbb{Z}[\Gamma ]$. If $N\in \mathbb{N},\,N\ge2$, and if $\Sigma \subset \Sigma _N=\{0,\ldots,N-1\}^{\Gamma }$ is  a subshift (i.e., a closed, left-shift-invariant subset), then the topological entropy $\mathsf{h}(\Sigma )$ is given by
	$$
\mathsf{h}(\Sigma)= \lim_{n\to\infty }\frac{ \log |P_{F_n}(\Sigma)| }{ | F_n | },
	$$
where $(F_n)_{n\ge1}$ is a F\o lner sequence in $\Gamma$ and $P_F\colon \Sigma \longrightarrow \{0,\dots ,N-1\}^F$ is the projection map for every subset $F\subset \Gamma$.
	\begin{enumerate}
	\item
A subshift $\Sigma \subset \Sigma _N$ is a \textit{symbolic cover} of $X_f$ if there exists a left-shift-equivariant continuous surjective map $\xi\colon \Sigma \longrightarrow X_f$;
	\item
The symbolic cover $\Sigma $ is an \textit{equal-entropy symbolic cover} of $X_f$ if the the left shift actions of $\Gamma $ on $\Sigma $ and on $X_f$ have the same topological entropy.
	\end{enumerate}
	\end{defi}
Before proving that $\Sigma _2=\{0,1\}^{\h}$ is an equal-entropy symbolic cover of $X_f$ for $f=2-x^{-1}-y^{-1}$ let us discuss some motivation for this claim.

	\begin{theo}[\cite{LS2}]
	\label{t:entropy}
Let $f\in \Z[\h]$ be of the form \eqref{eq:form1} -- \eqref{eq:form2} as in Section \ref{s:expansive}. Then the topological entropy of $(X_f,\alpha_f)$ satisfies that
	\begin{equation}
	\label{eq:entropy1}
\mathsf{h}(\alpha _f)\ge \int_{\s} \max\biggl (\int_{\s}\log |g_0(\xi,\theta)|\,d\xi , \int_\mathbb{S}\log |g_1(\xi,\theta)| d\xi \biggr) \,d\theta ,
	\end{equation}
where all integrals are taken with respect to normalised Lebesgue measure on $\mathbb{S}$. If $\alpha _f$ is expansive, then
	\begin{equation}
	\label{eq:entropy2}
\mathsf{h}(\alpha _f)=\int_{\s} \max\biggl (\int_{\s}\log |g_0(\xi,\theta)|\,d\xi , \int_\mathbb{S}\log |g_1(\xi,\theta)| d\xi \biggr) \,d\theta .
	\end{equation}
	\end{theo}

	\begin{theo}
	\label{t:cover}
Let $f=2-x^{-1}-y^{-1} \in \Z[\h]$. Then the full 2-shift $\Sigma _2=\{0,1\} ^{\h}$ is a symbolic cover of $X_f$.
	\end{theo}

	\begin{coro}
	\label{c:entropy}
Let $f=2-x^{-1}-y^{-1} \in \Z[\h]$. Then  $\mathsf{h}(\alpha_f)= \log 2$, so that $\Sigma _2$ is an equal-entropy symbolic cover of $X_f$.
	\end{coro}

	\begin{proof}[Proof of Corollary \ref{c:entropy}]
According to Theorem \ref{t:entropy}, $\mathsf{h}(\alpha _f)\ge \log 2$. Furthermore, since $\Sigma _2$ is a symbolic cover of $X_f$, $\mathsf{h}(\alpha _f)\le \log 2$.
	\end{proof}

	\begin{proof}[Proof of Theorem \ref{t:cover}]
If $\xi _g\colon \ell ^\infty (\h,\mathbb{Z})\longrightarrow X_f$ is the left-shift-equivariant homomorphism defined by Theorem \ref{t:mainmulti} and \eqref{eq:xig}, then Theorem \ref{t:codingmaptheo} (2) shows that $\xi _g(\Sigma _3)=X_f$, i.e., that $\Sigma _3$ is a symbolic cover of $X_f$. In order to prove that $\Sigma _2$ is also a symbolic cover of $X_f$ we use a modification of a toppling argument used in \cite{sv} in the process of showing that the $d$-dimensional critical sandpile model is an equal-entropy symbolic cover of the \textit{harmonic model}, the principal algebraic $\mathbb{Z}^d$-action defined by the graph-Laplacian of $\Z^d$.

Let $v \in \Sigma _3$. Every group element $\gamma\in\h$ can be written uniquely as a product $x^ly^mz^n$ with $(l,m,n)\in \Z^3$. For all $M>0$ let
	$$
A_M \coloneqq \{x^ly^mz^n \in \h : 0 \leq |l|\leq M, 0\leq |m| \leq M ,\, 0\leq |n|\leq M^2 \}.
	$$
We first perform topplings on $A_M$, defined by
	\begin{equation}
	\label{eq:toppling}
v\mapsto v-\gamma f^*\enspace \textup{for appropriate choices of} \enspace \gamma \in \h,
	\end{equation}
in such a way that the new configuration $\bar{v}_M$ has the following properties
	\begin{equation*}
(\bar{v}_M)_\gamma \in
	\begin{cases}
\{0,1\} & \text{for}\, \gamma \in A_M
	\\
\{0,\ldots, 4\} & \gamma \in \h \smallsetminus A_M .
	\end{cases}
	\end{equation*}
Note that topplings shift mass only in the positive $x$ and $y$ directions. The minimal number of topplings required to obtain such a point $\bar{v}$ is finite, and the finite subset of $\h$ which is affected by this toppling process is given by $\textup{supp}(v-\bar{v}_M)=\{\gamma \in \h:v_\gamma \ne \bar{v}_\gamma \}$. Let
	$$
B_M \coloneqq \textup{supp}(v-\bar{v}_M) \smallsetminus A_M \quad \text{and} \quad C_M \coloneqq \h \smallsetminus \textup{supp}(v-\bar{v}_M).
	$$
We know that $\xi_g (v)= \xi_g (\bar{v}_M)$ for all $M \geq 1$ since $v-\bar{v}_M \in (f^*)$. Next we define
	\begin{equation*}
(\tilde{v}_M)_\gamma \coloneqq
	\begin{cases}
(\bar{v}_M)_\gamma, & \text{for}\, \gamma \in A_M
	\\
v_\gamma & \text{otherwise.}
	\end{cases}
	\end{equation*}
Note that $(\tilde{v}_M)_\gamma=(\bar{v}_M)_\gamma$ for $\gamma \in A_M \cup C_M$ and $\textup{supp}(\tilde{v}_M - \bar v _M )\subset B_M$. Let $w=(1-z^{-1})^2 \cdot {f^*}^\sharp$ (cf. Theorem \ref{t:mainmulti}), and let and $u_M=\bar{v}_M - \tilde{v}_M$. Then
	\begin{equation*}
\bar{\xi}_g(u_M)_{\gamma'} = (u_M \cdot w)_{\gamma'} = \sum_{\gamma \in A_M} (u_M)_\gamma w_{\gamma^{-1}\gamma'}+\sum_{\gamma \in B_M} (u_M)_\gamma w_{\gamma^{-1}\gamma'}+\sum_{\gamma \in C_M} (u_M)_\gamma w_{\gamma^{-1}\gamma'}.
	\end{equation*}
The first and third sum do not contribute. For the second sum we make the following estimate for all $\gamma'\in\h$ and all $M >0$
	\begin{equation*}
\biggl | \sum_{\gamma \in B_M} (u_M)_\gamma w_{\gamma^{-1}\gamma'} \biggr | \leq \max_{\gamma \in B_M}|(u_M)_\gamma | \sum_{\gamma \in B_M} |w_{\gamma^{-1}\gamma'}| \leq 4 \sum_{\gamma \in \h \smallsetminus A_M} | w_{\gamma^{-1}\gamma'}|.
	\end{equation*}
The term on the right hand side of the last inequality goes to $0$ as $M \longrightarrow \infty$. From this we conclude that $\xi_g(\Sigma _2)$ is dense in $X_f$. Since $\Sigma _2$ is compact and $\xi _g$ is continuous by Theorem \ref{t:codingmaptheo} we conclude that $\xi_g(\Sigma _2) = X_f$.
	\end{proof}

	\section*{Appendix}
	\subsection*{The Proof of Theorem \ref{MainEstimate}}
	\label{ss:proof1}

Let $f=2-x-y$, and let ${f^*}^\sharp$ be defined by \eqref{eq:fundsol}. In order to find central multipliers of ${f^*}^\sharp$ (or, equivalently, of $f^\sharp$) we need to establish some properties of $q$-binomial coefficients. For this we recall the following theorem of Brunetti and Del Lungo \cite{BDL}.

	\begin{theo}
	\label{t:brunetti}
Let $\qbinom{n}{k}=\sum_{j=0}^{k(n-k)} c_j q^j$ be the q-binomial coefficient for $n\ge k\ge 0$ with $d= \gcd(n,k)$. Then
	$$
A_{n,k} = \qbinom{n}{k} \frac {1-q^d}{1-q^{n}} =\sum_{j} a_jq^j
	$$
is a polynomial with nonnegative integer coefficients.
	\end{theo}

This theorem was established first in \cite{Andrews} for the case $\gcd(n,k)=1$. From now on we assume that $k >0$. We extend Theorem \ref{t:brunetti} slightly.
	\begin{lemm}
	\label{Aispol}
If $n\geq k$, $p\geq 1$ are such that
	$$
A_{n,k,p}(q)\coloneqq \qbinom {n}{k} \frac {1-q}{1-q^p}
	$$
is a polynomial, then $A_{n,k,p}$ is a polynomial with nonnegative integer coefficients.
	\end{lemm}

	\begin{proof}
If $p=1$ there is nothing to prove. Assume $p>1$. The coefficients $\{c_j\}$ of $\qbinom{n}{k}$ and $\{a_j\}$ of $A_{n,k,p}$ satisfy the relation
	$$
a_j-a_{j-p} =c_j-c_{j-1}
	$$
for every $j$. Thus, for every $j$,
	$$
a_j=\sum_{t\ge 0} (c_{j-tp}-c_{j-tp-1}),
	$$
where $c_m=0$ for $m<0$. Using the unimodality property of the $\{c_j\}$, one can conclude that
	$$
a_j\ge 0\enspace \textup{for all}\enspace  j<\frac {k(n-k)}{2}.
	$$
Since $A_{n,k,p}(q)=\sum a_jq^j$ is also reciprocal, i.e. $A_{n,k,p}(\frac 1 q) = q^{-\textup{deg}(A_{n,k,p})} A_{n,k,p}(q)$, and the coefficients are centrally symmetric, i.e.,
	$$
\frac 12\bigl( k(n-k)-p+1\bigr)< \frac {k(n-k)}{2},
	$$
we can conclude that $A_{n,k,p}$ has nonnegative integer coefficients.
	\end{proof}

	\begin{coro}
	\label{Aestimate}
Under the assumption of the previous lemma, one has
	$$
\|A_{n,k,p}\|_1 = \frac 1p\binom{n}{k}.
	$$
and hence
	$$
\biggl\|\qbinom{n}{k} \cdot(1-q)\biggr\|_1\leq 2\|A_{n,k,p}\|_1 =\frac 2p \binom{n}{k}.
	$$
	\end{coro}

	\begin{lemm}
	\label{primedivqbin}
Suppose that $p> k$ is a prime dividing $\binom{n}{k}$. Then
	$$
A_{n,k,p}\coloneqq \qbinom {n}{k} \frac {1-q}{1-q^p}
	$$
is a polynomial.
	\end{lemm}

	\begin{proof}
One has
	$$
\qbinom{n}{k} =\prod_{j=1}^n \phi_j(q)^ {\lfloor\frac nj \rfloor-\lfloor\frac kj \rfloor-\lfloor\frac {n-k}j \rfloor } ,
	$$
 by \cite{KW}, where $\phi_j(q)$ is the $j$-th cyclotomic polynomial given by
	$$
\phi_j(q)= \prod_ {1 \leq k \leq j,\,\gcd (k,j)=1} (q-e^\frac{2 \pi i k}{j} ),
	$$
where each coefficient is either $0$ or $1$. Let
	$$
J_{n,k} =\biggl\{ j>1: \quad \Bigl\lfloor\frac nj \Bigr\rfloor-\Bigl\lfloor\frac kj \Bigr\rfloor-\Bigl\lfloor\frac {n-k}j \Bigr\rfloor=1\biggr\}.
	$$
Then
	$$
\qbinom{n}{k} =\prod_{j\in J_{n,k}} \phi_j(q).
	$$
Now recall Legendre's theorem :
	$$
n! = \prod_{p} p^{\lfloor\frac n{p}\rfloor+\lfloor\frac n{p^2}\rfloor+\,\cdots},
	$$
where the product runs over all primes $p$ with $p\leq n$. Then for every prime $p$, the maximal degree $\alpha$ such that $p^\alpha $ divides ${n\choose k}$ is given by
	$$
\alpha = \Bigl\lfloor\frac n{p} \Bigr\rfloor-\Bigl\lfloor\frac k{p} \Bigr\rfloor-\Bigl\lfloor\frac {n-k}{p} \Bigr\rfloor+ \Bigl\lfloor\frac n{p^2} \Bigr\rfloor-\Bigl\lfloor\frac k{p^2} \Bigr\rfloor-\Bigl\lfloor\frac {n-k}{p^2} \Bigr\rfloor+ \,\cdots
	$$
Suppose that $p$ is a prime $> k$ dividing $\binom{n}{k}$. We will show that
	\begin{equation}
	\label{need}
\Bigl\lfloor\frac n{p} \Bigr\rfloor-\Bigl\lfloor\frac k{p} \Bigr\rfloor-\Bigl\lfloor\frac {n-k}{p} \Bigr\rfloor= \Bigl\lfloor\frac n{p} \Bigr\rfloor-\Bigl\lfloor\frac {n-k}{p} \Bigr\rfloor=1 .
	\end{equation}
Since $p$ divides $\binom{n}{k}$, $p$ must divide at least one of the integers $n-k+1,\ldots, n$. Let $p$ divide $n-j$ for some $j\in\{0,\ldots, k-1\}$. Then the equation $n-j=cp$ for some $c\in\mathbb N$ implies that $n-k <cp$ and hence $\Bigl\lfloor\frac {n-k}{p} \Bigr\rfloor\le c-1$.

At the same time, $n=cp+j$ and hence $\lfloor\frac n{p} \rfloor\ge c$. Therefore \eqref{need} holds, and hence $p\in J_{n,k}$.

\smallskip For every $n\ge1$ we have that $q^n-1 = \prod_{j :\,j|n} \phi_j(q)$ (cf. \cite{KW}). In particular, if $p$ is a prime, then $q^p-1=\phi_p(q) (q-1)$. One has
	$$
\qbinom{n}{k} =\phi_p(q)\cdot \prod_{j\in J_{n,k}\smallsetminus\{p\}} \phi_j(q) = \frac{1-q^p}{1-q} \prod_{j\in J_{n,k}\smallsetminus\{p\}} \phi_j(q),
	$$
and hence
	$$
\qbinom{n}{k} \frac{1-q}{1-q^p}=\prod_{j\in J_{n,k}\smallsetminus\{p\}} \phi_j(q)
	$$
is a polynomial.
	\end{proof}

A classical theorem, discovered independently by J. Sylvester \cite{Sylvester} and Schur \cite{Schur}, states that the product of $k$ consecutive integers, each greater than $k$, has a prime divisor greater than $k$. A consequence of this fact is that, if $n\geq 2k$, then
	$$
\binom{n}{k}\text{ has a prime factor }p> k.
	$$
Even stronger (cf. \cite{Faulkner}): if $n\ge 2k$, then $ \binom{n}{k}\text{ has a prime factor }p\ge \frac 75 k $. Take $p$ to be the largest prime factor of $\binom{n}{k}$. Using Lemma \ref{primedivqbin} and Lemma \ref{Aispol} (Corollary \ref{Aestimate}), we conclude that
	\begin{equation}
	\label{estiesti}
\biggl\|\qbinom{n}{k}\cdot(1-q)\biggr\|_1 \le \frac 2p\binom{n}{k} \le \frac 2k\binom{n}{k}.
	\end{equation}

\medskip \subsection*{A central multiplier of the formal inverse of $f=2-x-y$}
	\label{ss:multiplier}

Let ${f^*}^\sharp$ be given by \eqref{eq:fundsol}. The estimate in (\ref{estiesti}) allows us to find the minimal central multiplier $g$ such that $g ^*\cdot {f^*}^\sharp \in \ell^1(\h)$.

	\begin{lemm}
For $n >0$,
	$$
\sum_{k=1}^{\lceil n/2\rceil} \frac 1k \binom{n}{k}=\mathcal O\Bigl( 2^n\frac { \log n}{n}\Bigr),
	$$
and
	$$
\mathsf S(n)\coloneqq \sum_{k=0}^n \biggl\|\qbinom{n}{k}\cdot(1-q)\biggr\|_1 =\mathcal O\Bigl( 2^n\frac { \log n}{n}\Bigr).
	$$
	\end{lemm}

	\begin{proof}
First, recall the discrete Chebyshev inequality: if $a_1\ge a_2\ge \ldots \ge a_m\ge 0$, $0\le b_1\le b_2\le \ldots \le b_m$, then for any $p_1,\ldots, p_m\ge 0$, one has
	$$
\sum_{k} p_k\sum_{k} p_k a_k b_k\le \sum_{k} p_k a_k\sum_{k} p_k b_k.
	$$
Let $p_k=1$, $a_k=1/k$, $b_k=\binom{n}{k}$, then
	$$
\sum_{k=1}^{\lceil n/2\rceil } \frac 1k \binom{n}{k} \le\frac 1{\lceil n/2\rceil } \sum_{k=1}^{\lceil n/2\rceil} \frac 1k \sum_{k=1}^{\lceil n/2\rceil} \binom{n}{k}
	$$
and the first result follows.

\smallskip Furthermore,
	\begin{align*}
\sum_{k=0}^n \biggl\|\qbinom{n}{k}\cdot(1-q)\biggr\|_1& =4+ \sum_{k=1}^{n-1} \biggl\|\qbinom{n}{k}\cdot(1-q)\biggr\|_1 \le 4+ 2\sum_{k=1}^{\lceil n/2\rceil} \biggl\|\qbinom{n}{k}\cdot(1-q)\biggr\|_1
	\\
&\le 4+ 4\sum_{k=1}^{\lceil n/2\rceil} \frac 1{k}\binom{n}{k}= \mathcal O\Bigl( 2^n\frac { \log n}{n}\Bigr). \qedhere\end{align*}
\renewcommand{\qedsymbol}{}
	\end{proof}

	\begin{prop}
	\label{p:2-x-y}
Let $f=2-x-y$. Then $(1-z^{-1})^2$ is a central multiplier of ${f^*}^\sharp$ in \eqref{eq:fundsol}.
	\end{prop}

	\begin{proof}
We recall the following identities, e.g \cite{KC}:
	$$
\qbinom{n+1}{k+1} =\frac{1-q^{n+1}}{1-q^{n-k}} \qbinom{n}{k+1}
	$$
and
	$$
\qbinom{n+1}{k+1} = \frac{1-q^{n-k+1}}{1-q^{k+1}} \qbinom{n+1}{k} .
	$$
Another important identity is the \textit{q-Vandermonde formula}
	$$
\qbinom{m+n}{k} = \sum_{j} \qbinom{m}{k-j} \qbinom{n}{j} q^{j(m-k+j)} .
	$$
The nonzero contributions to this sum come from values of $j$ such that the $q$-binomial coefficients on the right side are nonzero, that is, from
	$$
\max(0, k - m) \le j \le \min(n, k).
	$$
The Vandermonde $q$-binomial identity, for $k\ge j\geq 0$, gives us
	$$
\qbinom{2k}{k-j}=\sum_{i=0}^{k-j} q^{x(k,j,i)}\qbinom{k}{i}\qbinom{k}{k-j-i}= \sum_{i=0}^{k-j} q^{x(k,j,i)}\qbinom{k}{i}\qbinom{k}{i+j}.
	$$
Thus
	$$
\biggl\|\qbinom{2k}{k-j}\cdot(1-q)^{2}\biggr\| \le \sum_{i=0}^{k-j} \biggl\|\qbinom{k}{i}\cdot(1-q)\|_1 \biggl\|\qbinom{k}{i+j}\cdot(1-q)\biggr\|_1,
	$$
and hence,
	\begin{align*}
\sum_{j=0}^{k}&\biggl\|\qbinom{2k}{k-j}\cdot(1-q)^{2}\|_1\le \sum_{j=0}^{k} \sum_{i=0}^{k-j} \biggl\|\qbinom{k}{i}\cdot(1-q)\|_1 \biggl\|\qbinom{k}{i+j}\cdot(1-q)\biggr\|_1
	\\
&= \sum_{m=0}^k \biggl\|\qbinom{k}{m}\cdot(1-q)\biggr\|_1 \sum_{i=0}^m \biggl\|\qbinom{k}{i}\cdot(1-q)\biggr\|_1
	\\
&\le \sum_{m=0}^k \biggl\|\qbinom{k}{m}\cdot(1-q)\biggr\|_1 \sum_{i=0}^k \biggl\|\qbinom{k}{i}\cdot(1-q)\biggr\|_1 = \bigl( \mathsf S(k)\bigr)^2=\mathcal O\Bigl(2^{2k} \frac {\log^2 k}{k^2}\Bigr).
	\end{align*}
It follows that
	\begin{equation}
	\label{even}
\sum_{k} \frac 1{2^{2k}} \sum_{j=0}^{k}\biggl\|\qbinom{2k}{k-j}\cdot(1-q)^{2} \biggr\|_1<\infty,
	\end{equation}
and hence
	$$
\sum_{k} \frac 1{2^{2k}} \sum_{j=0}^{2k}\biggl\|\qbinom{2k}{j}\cdot(1-q)^{2} \biggr\|_1<\infty .
	$$
For $1\le j\le 2k$,
	$$
\qbinom{2k+1}{j} =\qbinom{2k}{j}+q^{2k+1-j}\qbinom{2k}{j-1}
	$$
Thus
	$$
\norm {\qbinom{2k+1}{j}\cdot(1-q)^2}\le \norm {\qbinom{2k}{j}\cdot(1-q)^2}+\norm {\qbinom{2k}{j-1}\cdot(1-q)^2}.
	$$
We conclude that $(1 - z^{-1})^2 \cdot {f^*}^\sharp \in \ell^1(\h)$.
	\end{proof}

\subsection{The inverse of $f=3+x+y+z$}
	\label{ss:3+x+y+z}

In this subsection we prove that $f=3+x+y+z$ is invertible in $\ell ^1(\h,\mathbb{R})$.

Using Taylor series expansion at $t=0$, one easily checks that
	\begin{equation}
\frac 1{(3+t)^k}= \sum_{n=0}^\infty \frac {(-1)^{n}} {3^{n+k}} {n+k-1\choose n} t^{n}.
	\end{equation}
Similarly,
	\begin{equation}
	\label{EulerTrans}
\frac 1{(3+t)^k}= \frac 1{(4+(t-1))^k}= \sum_{n=0}^\infty \frac {(-1)^{n}} {4^{n+k}} {n+k-1\choose n} (t-1)^{n}.
	\end{equation}
For $k\ge 1$, define $v^{(k)}\in \ell^\infty(\Z)$ as
	\begin{equation}
	\label{sol3z}
v^{(k)}_n=
	\begin{cases}
\frac {1} {3^{n+k}}{n+k-1\choose n}&\textup{if}\enspace  n\ge 0,
	\\
0&\textup{if}\enspace n<0.
	\end{cases}
	\end{equation}
One readily checks that $v^{(k)}=\sum_{n=0}^\infty v^{(k)}_n z^k$ is an inverse of $(3+z)^k$. By (\ref{EulerTrans}), one has
	$$
v^{(k)}=\sum_{n}^\infty v_{n}^{(k)} z^n=\sum_{n=0}^\infty \frac {(-1)^{n}} {4^{n+k}} {n+k-1\choose n} (z-1)^{n}.
	$$
Note that the $\ell^1$-norm of $ v^{(k)}$ satisfies that
	$$
\|v^{(k)}\|_1=\frac {1}{3^k}\sum_{n=0}^{\infty } \frac {1} {3^n}{n+k-1\choose n}= \frac {1}{3^k} \frac {1}{\bigl(1-\frac 13\bigr)^k}= \frac {1}{2^k}.
	$$
We are going to construct a formal inverse of $3+x+y+z$. One easily checks\footnote{The way to derive such expressions is as follows: $w=\frac 1{3+x+y+z} = \frac {1}{(3+z)}\cdot\frac {1} {1+\frac {x+y}{3+z}}=\sum_{M=0}^\infty (-1)^M(x+y)^M \cdot \frac {1}{(3+z)^{M+1}}=\sum_{M=0}^\infty (-1)^M(x+y)^M \cdot v^{(M+1)}$} that $f^\sharp=\sum_{M=0}^\infty (-1)^M(x+y)^M \cdot v^{(M+1)}$ is a formal inverse of $f$, where $v^{(M+1)}$ is given by (\ref{sol3z}). We have to show that $f^\sharp \in \ell ^1(\h,\mathbb{R})$.

By the $q$-binomial theorem,
	\begin{align*}
f^\sharp&=\sum_{M=0}^{\infty} \sum_{T=0}^M (-1)^M y^T x^{M-T} {M \brack T}_{q} v^{(M+1)} =\sum_{a,b=0}^{\infty} (-1)^{a+b} y^a x^b {a+b \brack b}_{q} v^{(a+b+1)}
	\end{align*}
where $q =z$. One needs relatively sharp estimates to prove the $\ell^1$-summability of $f^\sharp$. Note that
	$$
{a+b \brack b}_{q=z} v^{(a+b+1)}= \sum_{n=0}^\infty \frac {(-1)^{n}} {4^{n+a+b+1}} {n+a+b\choose n} (z-1)^{n}{a+b \brack b}_{q=z},
	$$
and hence
	\begin{equation}
	\label{estnormab}
\biggl\|{a+b \brack b}_{q=z} v^{(a+b+1)}\biggr\|_1\le \frac {1}{4^{a+b+1}} \sum_{n=0}^\infty {a+b+n\choose n} \frac {1}{4^n} \biggl\|(z-1)^n \cdot{a+b\brack b}_{q=z}\biggr\|_1.
	\end{equation}

Let us now proceed with estimating the norms
	$$
\biggl\|(z-1)^n\cdot{a+b\brack b}_{q=z} \biggr\|_1
	$$
For with $n=0,1$ we estimate as
	$$
\biggl\|{a+b\brack b}_{q=z} \biggr\|_1={a+b\choose b}, \quad \biggl\| (z-1)\cdot{a+b\brack b}_{q=z}\biggr\|_1\le 2 \biggl\| {a+b\brack b}_{q} \biggr\|_1=2{a+b\choose b}.
	$$
For $n\ge 2$ we proceed as follows:
	$$
\biggl\|(z-1)^n\cdot {a+b\brack b}_{q=z}\biggr\|_1= \biggl\|(z-1)^{n-2} \cdot (z-1)^2\cdot {a+b\brack b}_{q=z}\biggr\|_1 \le 2^{n-2} \biggl\|(z-1)^2\cdot {a+b\brack b}_{q=z}\biggr\|_1.
	$$
We can thus continue estimate in \eqref{estnormab} as follows
	\begin{align*}
N(a,b)&=\biggl\|{a+b \brack b}_{q=z}\cdot v^{(a+b+1)}\biggr\|_1\le \frac 1{4^{a+b+1}}{a+b\choose 0}{a+b\choose b}
	\\
&\qquad \qquad + \frac 2{4^{a+b+2}} {a+b+1\choose 1} {a+b\choose b}
	\\
&\qquad \qquad +\frac 1{4^{a+b+1}}\sum_{n=2}^\infty \frac 1{4^n}{a+b+n\choose n}\cdot 2^{n-2} \biggl\|(z-1)^2\cdot {a+b\brack b}_{q=z}\biggr\|_1
	\\
&\le \frac 1{4^{a+b+2}} {a+b\choose b} (a+b+1)
	\\
&\qquad \qquad + \biggl\|(z-1)^2\cdot{a+b\brack b}_{q=z}\biggr\|_1 \cdot \frac {{1}}{4^{a+b}}\sum_{n=2}^\infty {a+b+n\choose n} \left(\frac 12\right)^n
	\\
&\le \frac 1{4^{a+b+2}} {a+b\choose b} (a+b+1) + \frac {2}{2^{a+b+1}} \biggl\|(z-1)^2\cdot {a+b\brack b}_{q=z}\biggr\|_1.
	\end{align*}

Finally, putting all estimates together, we obtain that
	\begin{align*}
\|w\|_1&\le \sum_{a,b=0}^{\infty} \biggl\|{a+b \brack b}_{q} v^{(a+b+1)}\biggr\|_1
	\\
&\le \sum_{a,b=0}^\infty \frac {{a+b\choose b}}{4^{a+b+1}} (a+b+1) +2\sum_{a,b=0}^\infty \frac 1{2^{a+b+1}} \biggl\|(z-1)^2\cdot{a+b\brack b} _{q=z} \biggr\|_1.
	\end{align*}
The first sum is clearly finite. We also know from Theorem \ref{t:mainmulti} that the last sum is finite as well.

\subsection*{Remark}

We end the appendix with the following conjecture generalizing the results of \cite{Andrews} and \cite{BDL}. Let $\qbinom{n}{k}$ be the $q$-binomial coefficient for $n\ge 2k\ge 0$ with $\gcd(n,k)=1$. Then there exists an $n-k\leq m < n$, such that
	$$
B_{n,k} = \qbinom{n}{k} \frac {1-q}{1-q^{n}} \frac {(1-q)^2}{1-q^{m}} =\sum_{j} a_jq^j
	$$
is a polynomial with nonnegative integer coefficients. The complexity of our estimates would reduce drastically if the hypothesis would turn out to be true.

\end{document}